\documentclass{amsproc}

\usepackage{fancyhdr}
\usepackage{amsfonts}
\usepackage{amsmath}
\usepackage{amssymb}
\usepackage{amsthm}
\usepackage{tikz}
\usepackage{accents}
\usepackage{tikz-cd}

\usepackage{mathrsfs}
\usepackage{amscd}
\usepackage{amstext}

\usepackage{graphics}
\usepackage{graphicx}

\newtheorem{theorem}{Theorem}[section]
\newtheorem{lemma}[theorem]{Lemma}

\theoremstyle{definition}
\newtheorem{definition}[theorem]{Definition}

\newtheorem{proposition}[theorem]{Proposition}
\theoremstyle{remark}
\newtheorem{remark}[theorem]{Remark}

\numberwithin{equation}{section}

\newcommand{\re}{\mathbb{R}}\newcommand{\N}{\mathbb{N}}
\newcommand{\zz}{\mathbb{Z}}\newcommand{\C}{\mathbb{C}}

\newcommand{\Z}{{\zz}^d}

\newcommand{\R}{{\re}^d}

\newcommand{\cs}{{\mathcal S}}

\newcommand{\cf}{{\mathcal F}}
\newcommand{\cfi}{{\cf}^{-1}}

\newcommand{\supp}{{\rm supp \, }}

\newcommand{\gf}{\mathcal{F}}

\newcommand{\be}{\begin{equation}}
\newcommand{\ee}{\end{equation}}
\newcommand{\beq}{\begin{eqnarray}}
\newcommand{\beqq}{\begin{eqnarray*}}
\newcommand{\eeq}{\end{eqnarray}}
\newcommand{\eeqq}{\end{eqnarray*}}

\begin{document}

\title{Isotropic and Dominating Mixed Besov Spaces -- a Comparison}

\author{Van Kien Nguyen}
\address{Friedrich-Schiller-University Jena, Ernst-Abbe-Platz 2, 07737 Jena, Germany and University of Transport and Communications, 
Dong Da, Hanoi, Vietnam}
\email{kien.nguyen@uni-jena.de}

\author{Winfried Sickel}
\address{Friedrich-Schiller-University Jena, Ernst-Abbe-Platz 2, 07737 Jena, Germany}
\email{winfried.sickel@uni-jena.de}

\subjclass{Primary 46E35; Secondary 42B35}
\date{January 17, 2016}

\dedicatory{This paper is dedicated to the memory of Bj\"orn Jawerth.}

\keywords{distribution spaces, isotropic Sobolev spaces, 
Sobolev spaces of dominating mixed smoothness,  isotropic Besov spaces, Besov spaces of dominating mixed smoothness, embeddings.}

\begin{abstract}
We compare Besov spaces with isotropic smoothness with Besov spaces of dominating mixed smoothness.
Necessary and sufficient conditions for continuous embeddings will be given. 
\end{abstract}

\maketitle

\section{Introduction}


For  $t \in \N_0$ the isotropic Sobolev space $W^t_2(\R)$, built on $L_2 (\R)$,
is the collection of all functions $f \in L_2 (\R)$ such that
\[
\|f|W^t_2(\R)\|:=\sum_{|\bar{\alpha}| \leq t}\|D^{\bar{\alpha}}f|L_2(\R)\|<\infty\, .
\] 
The  Sobolev space of dominating mixed smoothness $S^t_2 W(\R)$
is the tensor product of the univariate Sobolev spaces $W^t_2(\re)$, with other words
\beqq
S^t_2W(\R):= \Big\{f\in L_2(\R): \|f|S^t_pW(\R)\|:=\sum_{\|\bar{\alpha}\|_{\infty} \leq t}\|D^{\bar{\alpha}}f|L_2(\R)\|<\infty\Big\}.
\eeqq 
Here $\bar{\alpha}=(\alpha_1,...,\alpha_d)\in \N_0^d$, $|\bar{\alpha}|=\alpha_1+...+\alpha_d$ and 
$\|\bar{\alpha}\|_{\infty}=\max_{i=1,...,d}|\alpha_i|$. Observe that 
the mixed derivative $D^{(t, \ldots \, t)}f$ has the highest order in this norm which is the reason for the name of these spaces.
Spaces of dominating mixed smoothness have found applications in approximation theory since the early sixties,  more recent in high dimensional approximation and 
information based complexity, see, e.g.,  \cite{T93} and \cite{NoWo08,NoWo10,NoWo12}.  
Obviously we have the chain of continuous embeddings
\beqq
W^{td}_2(\R) \hookrightarrow S^t_2W(\R) \hookrightarrow W^t_2(\R).
\eeqq
Also easy to see is the optimality of these embeddings in various directions. We will discuss this below.
These two types of Sobolev spaces  $W^t_2(\R)$ and $S^t_2W(\R)$ represent particular cases of corresponding scales of Besov spaces, denoted by 
$B^t_{p,q} (\R)$ (isotropic smoothness) and  $S^t_{p,q} B(\R)$ (dominating mixed smoothness).
Indeed, we have
\[
W^t_2(\R) = B^t_{2,2}(\R) \qquad \mbox{and}\qquad S^t_2W(\R) = S^t_{2,2}B(\R) 
\]
in the sense of equivalent norms.
In this paper we address the question under  which conditions on $t,p,q$ the embedding 
\be\label{ws-01}
 B^{td}_{p,q}(\R)\hookrightarrow S^t_{p,q} B(\R) \hookrightarrow B^t_{p,q}(\R)
 \ee
holds true. In addition we shall discuss the optimality of these embeddings in various directions. 
Let us mention here that Schmeisser \cite{Sc2} and Hansen \cite{Hansen}
have considered those embeddings as well. Below we will make a more detailed  comparison.
\\
Nowadays isotropic Besov spaces represent  a well accepted regularity notion in various fields of mathematics.
Besov spaces of dominating mixed smoothness are of increasing importance in approximation theory and information based complexity,  
we refer to \cite{T93}.  
As a special case, the scale $S^t_{p,p} B(\R)$  contain the tensor products of the univariate Besov spaces $B^t_{p,p}(\re)$, see \cite{SU09,SU10}.
It is the main aim of this paper to give a detailed comparison of these different extensions of univariate Besov spaces
into the multi-dimensional situation.
\\
Since a few years there is some strong motivation to study those spaces also for parameters $p,q<1$.
Let $\Phi := (\psi_j)_j$ denote a wavelet basis satisfying some additional 
smoothness, integrability,  and moment conditions. 
We consider best $m$-term approximation with respect to $\Phi$, i.e., we investigate the quantity
\[
\sigma_m (f,\Phi)_X := \inf \bigg\{
\|\, f - \sum_{j \in \Lambda} c_j \, \psi_j\,|\,X  \|:
\  |\Lambda|\le m\, , \ c_j \in \C\, , \, j \in \Lambda 
\bigg\}\, , \quad m \in \N_0\, .
\]
Associated widths are defined as follows. Let $X$ and $Y$ be quasi-Banach spaces such that 
$Y \hookrightarrow X$. Then we define
\[
\sigma_m (Y,X,\Phi):= \sup \, \Big\{\sigma_m (f,\Phi)_X : \quad 
\| \, f\, |Y\|\le 1\Big\}\, , \quad m \in \N_0\, . 
\]
Usually one concentrates on $X=L_p (\R)$. 
Here we would like to recall the breakthrough result of DeVore, Jawerth and Popov \cite{DJP}.
Let $0 < \tau < p$. A function $f$ belongs to the Besov space $B^{d(\frac{1}{\tau} - \frac{1}{p})}_{\tau,\tau}(\R)$
if, and only if  it satisfies
\[
\Big(\sum_{m=1}^\infty m^{-1}\,  \Big[m^{\frac{1}{\tau} - \frac{1}{p}} \, 
\sigma_m (f,\Phi)_{L_p (\R)} \Big]^\tau\Big)^{1/\tau}<\infty\, .
\] 
Hence, Besov spaces (in particular  with $\tau<1$) describe approximation spaces with respect 
to nonlinear approximation where the original question 
(behaviour of $\sigma_m (f,\Phi)_{L_p (\R)}$) has been asked for $L_p$-spaces with $p \ge 1$.
For further results in this directions supporting the importance of spaces with $p,q<1$, we refer 
to Jawerth and Milman \cite{JM}, \cite{JM2}.
Similar descriptions of  $S^{d(\frac{1}{\tau} - \frac{1}{p})}_{\tau,\tau}B(\R)$ exist as well, see \cite{HS1}, \cite{HS2}.
For us this motivates the investigation of our problem also for $p,q<1$.

\noindent
The paper is organized as follows. In Section \ref{def} we recall the definition of the spaces $B^t_{p,q}(\R)$ and $S^t_{p,q}B(\R)$.  
Our main results are stated in Section \ref{main}. Proofs are concentrated in Section \ref{proofs}.


\subsection*{Notation}


As usual, $\N$ denotes the natural numbers, $\N_0 := \N \cup \{0\}$,
$\zz$ the integers and
$\re$ the real numbers, $\C$ refers to the complex numbers. For a real number $a$ we put $a_+ := \max(a,0)$.
If $\bar{k} \in \N_0^d$, i.e., if $\bar{k}=(k_1, \ldots \, , k_d)$, $k_\ell \in \N_0$, $\ell=1, \ldots \, , d$, then we put
\[
|\bar{k}| := k_1 + \ldots \, + k_d\, .
\]
For $x \in \R$ we use $\|x\|_\infty:= \max_{j=1, \ldots \, d} \, |x_j|  $. 
The symbols  $c,c_1, c_2, \, \ldots  \, ,C, C_1,C_2, \,  \ldots $ denote   positive
constants which are independent of the main parameters involved but
whose values may differ from line to line. The symbol 
$A \asymp B$ means that there exist positive constants $C_1$ and $C_2$ such that $C_1\, A\leq B\leq C_2\, A.$\\
Let $X$ and $Y$ be  two quasi-Banach spaces. Then the symbol $X \hookrightarrow Y$ indicates that the embedding is continuous. 
By $C_0^\infty (\R)$ the set of compactly supported infinitely differentiable functions $f:\R \to \C$ is denoted.
Let $\cs (\R)$ be the Schwartz space of all complex-valued rapidly decreasing infinitely differentiable  functions on $\R$. 
The topological dual, the class of tempered distributions, is denoted by $\cs'(\R)$ (equipped with the weak topology).
The Fourier transform on $\cs(\R)$ is given by 
\[
\cf \varphi (\xi) = (2\pi)^{-d/2} \int_{\R} \, e^{ix \xi}\, \varphi (x)\, dx \, , \qquad \xi \in \R\, .
\]
The inverse transformation is denoted by $\cfi $.
We use both notations also for the transformations defined on $\cs'(\R)$.\\
Let $0 < p,q \leq \infty$. For an arbitrary countable index set $I$ we put
\[
\| (f_k)_{k\in I} |\ell_q(L_p)\| := \bigg(\sum_{k\in I} \Big(\int\limits_{\R} 
|f_k(x) |^p \,  dx\Big)^{q/p} \bigg)^{1/q} 
\]
(usual modification if $\max (p,q) = \infty$).


\section{Besov spaces of isotropic and dominating mixed smoothness}\label{def}


\subsection{Isotropic Besov spaces}

Usually these spaces are defined by using the modulus of smoothness.
Here we prefer to use the Fourier analytic descriptions
since we are dealing also with negative smoothness.
\\ 
Let $\phi_0\in C_0^\infty(\R)$ be a non-negative function such that  $\phi(x)=1$ if $|x|\leq 1$ and $\phi (x)=0$ if $|x|\geq \frac{3}{2}$. For $j\in \N$ we define 
\beqq
\phi_j(x):=\phi_0(2^{-j}x)-\phi_0(2^{-j+1}x), \qquad x \in \R\, .
\eeqq
This yields 
\[
\sum_{j=0}^\infty \phi_j(x) = 1 \qquad \mbox{for all}\quad x \in \R\, .
\]
We shall call $(\phi_j)_{j=0}^\infty$ a smooth dyadic decomposition of unity.

\begin{definition}\label{ibf} \rm Let $t \in \re$ and $0< p, q\leq \infty$.
Then $B^{t}_{p,q}(\R)$ is the collection of all
$f\in S'(\R)$ such that
 \beqq
            \|f|B^{t}_{p,q}(\R)\|^{\phi}  =
            \bigg(\sum\limits_{j=0}^{\infty}2^{jtq}
            \big\|\gf^{-1}[\phi_{j}\gf f](\cdot)|L_p(\R)
            \big\|^q\bigg)^{1/q}
            \label{ibnorm}
\eeqq
is finite (modification if $q=\infty$).
\end{definition}

\begin{remark}
\rm
Besov spaces are discussed in various monographs, let us refer to 
\cite{BL}, \cite{Ni}, \cite{Pe} and \cite{Tr83}.
They are quasi-Banach spaces (Banach spaces if $\min(p,q)\ge 1$) and 
they do not depend on the chosen generator $\phi_0$ of the smooth dyadic decomposition
(in the sense of equivalent quasi-norms).
We call them isotropic because they are invariant under rotations.
Characterizations in terms of differences can be found at various places, see, e.g., \cite[2.5]{Tr83} or \cite[3.5]{Tr92}.
\end{remark}

For us it will be convenient to switch to an equivalent quasi-norm. Let $\psi_0\in C_0^\infty (\R)$ such that 
 \beqq
 \psi_0(x)=1 \text{\ if\ } \sup_{i=1,...,d}|x_i|\leq 1
\quad \text{and}\quad  \psi_0(x)=0 \ \text{ if }\ \sup_{i=1,...,d}|x_i|\geq \frac{3}{2}\, .
\eeqq
 For $j\in \N$, we define 
 \be\label{unity1}
 \psi_j(x):=\psi_0(2^{-j}x)-\psi_0(2^{-j+1}x)\, , \qquad x \in \R\, .
 \ee
 Then we have
 \beqq
 \supp\psi_j \subset \{x:\ \sup_{i=1,...,d}|x_i|\leq 3.2^{j-1}\}\, \setminus \, \{x: \ \sup_{i=1,...,d}|x_i|\leq 2^{j-1}\}, \, \quad j \in \N\, .
 \eeqq
 
 As an easy consequence of \cite[Proposition 2.3.2]{Tr83} one obtains the following.
 
 \begin{proposition}
Let $t \in \re$ and $0<p, q\leq \infty$.
Then $B^{t}_{p,q}(\R)$ is the collection of all $f\in S'(\R)$ such that
      \beqq
            \|f|B^{t}_{p,q}(\R)\|^{\psi}  =
            \bigg(\sum\limits_{j=0}^{\infty}2^{jtq}
            \big\|\gf^{-1}[\psi_{j}\gf f](\cdot)|L_p(\R)
            \big\|^q\bigg)^{1/q}
            \label{ibnorm1}
        \eeqq
         is finite (modification if $q=\infty$). The quasi-norms $\|f|B^{t}_{p,q}(\R)\|^{\psi}$ and $\|f|B^{t}_{p,q}(\R)\|^{\phi}$ are equivalent.
\end{proposition}

In what  follows we will work with the $\psi-$norm. Therefore we shall write $\|f|B^{t}_{p,q}(\R)\| $ instead of $\|f|B^{t}_{p,q}(\R)\|^{\psi}$.


\subsection{Besov spaces of dominating mixed smoothness}


Let $\varphi_0 \in C_0^{\infty}({\re})$ satisfy  $\varphi_0(\xi) = 1$ on $[-1,1]$ and $\supp\varphi \subset [-\frac{3}{2},\frac{3}{2}]$. For $j\in \N$ we define
\be\label{unity2}
\varphi_j(x) = \varphi_0(2^{-j}x)-\varphi_0(2^{-j+1}x)\, , \qquad x \in \re\, .
\ee
Hence, $(\varphi_j)_{j=0}^\infty$ forms a smooth dyadic decomposition of unity on $\re$.
Now we switch to tensor products.    For $\bar{k} = (k_1,...,k_d) \in {\N}_0^d$ the function
    $\varphi_{\bar{k}}(x) \in C_0^{\infty}(\R)$ is defined by
$$
\varphi_{\bar{k}}(x) := \varphi_{k_1}(x_1)\cdot \, \ldots \, \cdot  \varphi_{k_d}(x_d)\, ,\qquad x\in \R.
$$

\begin{definition}\label{sprd} 
Let $t \in \re$ and $0<p,q\leq \infty$.
Then $S^{t}_{p,q}B(\R)$ is the collection of all $f\in S'(\R)$ such that
\beqq
            \|f|S^{t}_{p,q}B(\R)\|^{\varphi}  =
            \bigg(\sum\limits_{\bar{k}\in{\N}_0^d}2^{|\bar{k}|tq}
            \big\|\gf^{-1}[\varphi_{\bar{k}}\gf f](\cdot)|L_p(\R)
            \big\|^q\bigg)^{1/q}
            \label{bnorm}
\eeqq
is finite (modification if $q=\infty$).
    \end{definition}

\begin{remark}\label{blabla}
\rm
{\rm (i)} Besov spaces of dominating mixed smoothness are discussed in the  monographs Amanov \cite{Am} and 
Schmeisser, Triebel \cite{ST}, see also the booklet Vybiral \cite{Vybiral}.
They are quasi-Banach spaces (Banach spaces if $\min(p,q)\ge 1$) and 
they do not depend on the chosen generator $\phi_0$ of the smooth dyadic decomposition
(in the sense of equivalent quasi-norms). For characterizations in terms of differences we refer to \cite[2.3.4]{ST} and 
\cite{U1}.
\\
{\rm (ii)} For us of certain importance will be the following observation. Besov spaces of dominating mixed smoothness have a cross-quasi-norm, i.e., 
if $f_j \in B^t_{p,q}(\re)\, ,\  j=1, \ldots \, , d \, $ then 
\[
 f(x) = \prod_{j=1}^d f_j (x_j)  \in S^t_{p,q}B(\R) \quad \text{and} \quad
 \| \, f \, | S^t_{p,q}B(\R)\| = \prod_{j=1}^d \|\, f_j \, |B^t_{p,q} (\re)\| \, .
\]
{\rm (iii)} For $d=1$ we have $ S^{t}_{p,q} B(\re) = B^t_{p,q}(\re)\,  . $
\end{remark}


\section{The main results}\label{main}


We discuss these embeddings in \eqref{ws-01} separateley.


\subsection{The embedding of dominating mixed spaces into isotropic spaces}


\begin{theorem}\label{besov4}
 Let $0<p,q \le  \infty$ and $t\in \re$. Then we have 
\be\label{ws-04}
S^{t}_{p,q}B (\R) \hookrightarrow B^t_{p,q} (\R) 
\ee
if and only if one of the following conditions is satisfied 
\begin{itemize}
 \item $t>0$;
 \item $t=0$, $0 <p <\infty$ and $0 < q \le \min(p,2)$;
 \item $t=0$, $p=\infty$ and $q \le 1$.
\end{itemize}

\end{theorem}

\begin{remark}\label{incomparable1}
 \rm
Sufficient conditions for embeddings as in \eqref{ws-04} have been considered by Schmeisser \cite{Sc2} and Hansen \cite{Hansen}. 
Both used different methods than we do.
Schmeisser used characterizations by differences and concentrated on the Banach space case.  
Hansen showed $S^{t+ \varepsilon}_{p,q_0}B (\R) \hookrightarrow B^t_{p,q} (\R)$
with $\varepsilon >0$ and $q_0,q$ arbitrary by applying wavelet characterizations.
\end{remark}

We summarize what is known about the relation of 
$B^{t}_{p,q}(\R)$ and $S^t_{p,q} B(\R) $ in the remaining cases.
For subsets $X,Y$ of $\cs'(\R)$ we shall call not comparable if
\[
 X \setminus Y \neq \emptyset \qquad \mbox{and}\qquad Y\setminus X \neq \emptyset\, .
\]

\begin{proposition}\label{besov14}
{\rm (i)}  Let $1 \le p \le  \infty$, $0<q\leq \infty$ and $t<0$. Then we have 
$$
 B^{t}_{p,q}(\R)\hookrightarrow S^t_{p,q} B(\R) \, .
$$
{\rm (ii)} Let $t=0$, $p \neq 2$,  $1 <p<\infty$ and $\min (2,p)< q < \max(2,p)$. 
Then $B^{0}_{p,q}(\R)$ and $S^0_{p,q} B(\R) $ 
are not comparable.
\\
{\rm (iii)} Let $t=0$, $p=1$ and   $1 <q<\infty$.  
Then $B^{0}_{1,q}(\R)$ and $S^0_{1,q} B(\R) $ 
are not comparable.
\\
{\rm (iv)} Let $t=0$, $p = \infty$ and   $1 <q<\infty$.  
Then $B^{0}_{\infty,q}(\R)$ and $S^0_{\infty,q} B(\R) $ 
are not comparable.
\\
{\rm (v)} Let $t<0$,   $0 <p<1$ and $0 <  q \le \infty$. 
Then $B^{t}_{p,q}(\R)$ and $S^t_{p,q} B(\R) $ 
are not comparable.
\end{proposition}
$$
\begin{tikzpicture}
\draw[->, ](0,0) -- (7,0);
\draw[->, ] (0,-2 ) -- (0,2);
\draw (3.5,-2)-- (3.5,0);
\draw[->, ] (7.6,1.6) -- (6.3,0.1);
\node [] at (7.6,1.9) {{\rm critical line}};

\node[below] at (-0.2,0) {$0$};
\node [left] at (7.2,-0.4) {$\frac{1}{p}$};
\node [left] at (0,2) {$t$};
\node [] at (3.5,1.2) {$ S^{t}_{p,q}B(\R)  \hookrightarrow B^t_{p,q}(\R)$};

\node [] at (1.7,-1.3) {{\small  $ B^{t}_{p,q}(\R) \hookrightarrow S^{t}_{p,q}B(\R)$}};

\node [] at (5,-0.9) {{\rm not comparable}};
\fill (3.5,0) circle (2pt);
\fill (0,0) circle (2pt);
\node [left] at (3.5,-0.3) {1};
\path[draw, line width=1.5pt](0,0) -- (6.5,0);
\node [right] at (3,-2.5) {\textsc{Figure 1.}};
\end{tikzpicture}
$$
The embedding \eqref{ws-04} is optimal in the following sense.

\begin{theorem}\label{besov5}
 Let $0<p_0,p, q_0,q \le  \infty$ and $t_0,t\in \re$. 
Let $p,q$ and $t$ be fixed. 
 Within all spaces   $S^{t_0}_{p_0,q_0}B (\R)$ satisfying
\beqq
S^{t_0}_{p_0,q_0}B (\R) \hookrightarrow B^{t}_{p,q} (\R) 
\eeqq
the class $S^{t}_{p,q}B (\R)$ is the largest one.
\end{theorem}

\begin{remark}
 \rm Comparing Theorems \ref{besov4} and \ref{besov5} it is natural to ask also for the optimality  
 of $S^{t}_{p,q}B (\R) \hookrightarrow B^t_{p,q} (\R) $ in the other direction, i.e., we fix 
 $S^{t}_{p,q}B (\R)$ and look for spaces $ B^{t_0}_{p_0,q_0} (\R)$ such that \eqref{ws-04} is true.
For this we consider a special situation.  
Theorem \ref{besov4} yields  $S^{2}_{1,2}B (\R) \hookrightarrow B^2_{1,2} (\R) $.
On the other hand, a Sobolev-type embedding and Theorem \ref{besov4} imply
\[
 S^{2}_{1,2}B (\R) \hookrightarrow S^{3/2}_{2,2}B (\R) \hookrightarrow B^{3/2}_{2,2} (\R) \, , 
\]
see the comments at the beginning of Subsection \ref{proof5}.
But for $d \ge 2$ these isotropic Besov spaces  $B^2_{1,2} (\R)$ and $ B^{3/2}_{2,2} (\R)$ are not comparable.
Hence, an optimality in such a wide sense is not true. 
\end{remark}


\subsection{The embedding of isotropic spaces  into dominating mixed spaces}


\begin{theorem}\label{besov1}
 Let $0<p,q \le  \infty$ and $t\in \re$. Then we have 
\be\label{ws-03}
B^{td}_{p,q} (\R) \hookrightarrow S^{t}_{p,q}B (\R)  
\ee
if and only if one of the following conditions is satisfied 
\begin{itemize}
 \item $t> \max(0, \frac 1p -1)$;
 \item $t=0$, $1< p   \le \infty$ and $\max(2,p) \le q\le \infty$; 
 \item $0 <p\le  1$, $t= \frac 1p -1$ and $q=\infty$.
\end{itemize}
\end{theorem}

\begin{remark}
 \rm
Again we have to refer to Hansen \cite{Hansen} for an earlier result in this direction.
He proved  $S^{td+ \varepsilon}_{p,q_0}B (\R) \hookrightarrow S^t_{p,q}B (\R)$
with $\varepsilon >0$ and $q_0,q$ arbitrary.
\end{remark}

Also in this situation we summarize what is known about the relation of 
$B^{td}_{p,q}(\R)$ and $S^t_{p,q} B(\R) $ in the remaining cases.

\begin{proposition}\label{besov15}
{\rm (i)}  Let  $0<p,q\leq \infty$ and $t<0$. Then we have $$
S^{t}_{p,q}B (\R) \hookrightarrow B^{td}_{p,q} (\R)\,.
$$
{\rm (ii)} Let $0 <p<1$, $0< t < \frac 1p - 1$ and $0 <  q \le \infty$. 
Then $B^{t}_{p,q}(\R)$ and $S^t_{p,q} B(\R) $ are not comparable.
\\
{\rm (iii)} Let $0 <p<1$, $t=0$ and $0 <  q \le p$. 
Then $S^0_{p,q} B(\R)  \hookrightarrow B^{0}_{p,q}(\R)$ follows.
\\
{\rm (iv)} Let $0 <p<1$, $t=0$  and $p <  q \le \infty$. 
Then $B^{0}_{p,q}(\R)$ and $S^0_{p,q} B(\R) $ are not comparable.
\end{proposition}
\begin{remark}
 \rm
Obviously the case $t=0$ is covered by Proposition \ref{besov14}. 
The set  $\{(p,0): \: 1 \le p \le \infty\}$ is part of the critical line  of \eqref{ws-03}. 
Also for us it was surprising that the critical line for $0 <p<1$ is given by $\frac 1p -1$. 
\end{remark}
$$
\begin{tikzpicture}
\draw[->, ](0,0) -- (7,0);
\draw[->, ] (0,-2) -- (0,2);
\draw[->, ] (7.5,2) -- (5.5,1.5);
\node [] at (8,2.2) {{\rm critical line}};

\node[below] at (-0.2,0) {$0$};
\node  at (3.5,-0.3) {1};
\node [right] at (5,2.2) {$t=\frac{1}{p}-1$};

\node [left] at (7.2,-0.4) {$\frac{1}{p}$};
\node [left] at (0,2) {$t$};

\node [] at (2.3,1.2) {$ B^{td}_{p,q}(\R)  \hookrightarrow S^{t}_{p,q}B(\R)$};

\node [] at (3.5,-1.3) {{$ S^{t}_{p,q}B(\R) \hookrightarrow B^{td}_{p,q}(\R)$}};
\node [] at (6,0.7) {{\rm not comparable}};

\fill (0,0) circle (2pt);
\fill (3.5,0) circle (2pt);
\path[draw, line width=1.5pt](0,0) -- (3.5,0);
\path[draw, line width=1.5pt](3.5,0) -- (5.7,2);
\node [right] at (3,-2.5) {\textsc{Figure 2.}};
\end{tikzpicture}
$$

These embeddings are optimal in the following sense.

\begin{theorem}\label{besov2}
 Let $0<p_0,p, q_0,q \le  \infty$ and $t_0,t\in \re$. 
Let $p,q$ and $t$ be fixed. 
 Within all spaces   $B^{t_0}_{p_0,q_0} (\R)$ satisfying
\beqq
B^{t_0}_{p_0,q_0} (\R) \hookrightarrow S^{t}_{p,q} B(\R) 
\eeqq
the class $B^{td}_{p,q} (\R)$ is the largest one.
\end{theorem}

\begin{theorem}\label{besov3}
 Let $0<p_0,p, q_0,q \le  \infty$ and $t_0,t\in \re$. 
Let $p,q$ and $t$ be fixed. 
 Within all spaces   $S^{t_0}_{p_0,q_0}B (\R)$ satisfying
\beqq
B^{td}_{p,q} (\R) \hookrightarrow S^{t_0}_{p_0,q_0} B(\R) 
\eeqq
the class $S^{t}_{p,q} B(\R)$ is the smallest one.
\end{theorem}

\begin{remark}
 \rm
 Let us come back to the chain of  embeddings 
\beqq
W^{td}_2(\R)=B^{td}_{2,2} (\R) \hookrightarrow S^t_2W(\R) = S^{t}_{2,2} B(\R)  \hookrightarrow W^t_2(\R)= B^{t}_{2,2} (\R)
\eeqq
discussed in the Introduction. Employing Theorems \ref{besov5}, \ref{besov2} and \ref{besov3} we obtain the following optimality assertions.
\begin{itemize}
 \item Within all spaces   $S^{t_0}_{p_0,q_0}B (\R)$ satisfying $S^{t_0}_{p_0,q_0}B (\R) \hookrightarrow W^{t}_{2} (\R)$ 
the class $S^{t}_{p,q}B (\R)$ is the largest one.
\item
 Within all spaces   $B^{t_0}_{p_0,q_0} (\R)$ satisfying
$ B^{t_0}_{p_0,q_0} (\R) \hookrightarrow S^{t}_{2}W(\R) $
the class $B^{td}_{2,2} (\R) = W^{td}_2 (\R)$ is the largest one.
\item
Within all spaces   $S^{t_0}_{p_0,q_0}B (\R)$ satisfying
$W^{td}_{2} (\R) \hookrightarrow S^{t_0}_{p_0,q_0} B(\R) $
the class $S^{t}_{2} W(\R)$ is the smallest one.
\end{itemize}

\end{remark}


\section{Proofs}\label{proofs}


To prove our main results we will apply essentially four different tools:  
Fourier multipliers;  complex interpolation; assertions on dual spaces; some  test functions. 
In what follows we collect what is needed.


\subsection{Fourier multipliers}


Let us recall some Fourier multiplier assertions. 
For a compact subset $\Omega \in \R$ we introduce the notation
\beqq
L_p^{\Omega}(\R)=\{f\in \cs'(\R): \ \supp \gf f\subset \Omega\, ,\ f\in L_p(\R)\} \, .
\eeqq
For those spaces improved convolution inequalities hold, see  \cite[Proposition 1.5.1]{Tr83}.

\begin{lemma}\label{mul}
Let $\Omega$  and $\Gamma$ be compact subsets of $\R$. Let $0<p\leq \infty$ and put  $u:=\min(p,1)$. Then there exists a positive constant $c$ such that
\beqq
\| \gf^{-1}M\gf f|L_p(\R)\|\leq c\, \|\gf^{-1}M|L_u(\R)\|\cdot \|f|L_p(\R)\|
\eeqq
holds for all $f\in L_p^{\Omega}(\R)$ and all $\gf ^{-1}M\in L_u^{\Gamma}(\R)$.
\end{lemma}

Lemma \ref{mul} and the homogeneity properties of the Fourier transform yield the following.

\begin{lemma}\label{mul3} 
Let $(\psi_j)_{j=0}^\infty$ and $(\varphi_{\bar{k}})_{\bar{k} \in\N_0^d}$ be the two decompositions of unity defined in \eqref{unity1} and \eqref{unity2}, respectively.
Let $u=\min(1,p)$. 
Then there exists a positive constant $C$  such that
\be\label{ct3}
\| \, \gf^{-1}(\psi_j\varphi_{\bar{k}}\gf f)\, |L_p(\R)\| \leq C\, \|\gf^{-1}\varphi_{\bar{k}}\gf f|L_p(\R)\| 
\ee
and
\be\label{ct4}
\| \gf ^{-1}(\varphi_{\bar{k}}\psi_j\gf f)|L_p(\R)\| \leq C\, 2^{(jd-|\bar{k}|)(\frac{1}{u}-1)}\|\gf ^{-1}\psi_j \gf f|L_p(\R)\| 
\ee
hold for all $j \in \N_0$, all $k \in \N_0^d$  and all $f \in \cs'(\R)$ with finite  right-hand sides.
\end{lemma}

\begin{proof}
For $\bar{k}\in \N_0^d$ and $j \in \N_0$ we put
\beqq
\Omega_{\bar{k}}&=&\{x\in \R: |x_i|\leq 2^{k_i+1},\ i=1,...,d\},\\
\Gamma_j&=&\{x\in \R: \sup_{i=1,...,d}|x_i|\leq 2^{j+1}\}.
\eeqq
{\it Step 1.} Proof of \eqref{ct3}. For $f \in \cs'(\R)$ we put $g:=\gf ^{-1}\varphi_{\bar{k}}\gf f $. 
Then we have $g\in L_p^{\Omega_{\bar{k}}}(\R)$ and $g(2^{-j}\cdot)\in L_p^{\Gamma_0}(\R)$. Observe that
\begin{equation}
\begin{split}\nonumber
\|\gf ^{-1}\psi_j\gf g|L_p(\R) \| &=\, 2^{-\frac{jd}{p}}\| (\gf ^{-1}\psi_j\gf g)(2^{-j}\cdot)|L_p(\R)\|\\
&=\, 2^{-\frac{jd}{p}}\| \gf ^{-1}\big(\psi_j(2^{j}\cdot) \gf [g(2^{-j}\cdot)]\big)|L_p(\R)\|.
\end{split}
\end{equation}
Let $j \in \N$.
Lemma \ref{mul} together with $\supp\psi_{j}(2^{j}\cdot)\subset \Gamma_0$  yield
\begin{equation}
\begin{split}\nonumber
\|\gf ^{-1}\psi_j\gf g|L_p(\R) \|
&\leq\, c\,  2^{-\frac{jd}{p}} \| \gf ^{-1} (\psi_1 (2\, \cdot \, ))\, |L_{u}(\R)\|\cdot \|g(2^{-j}\cdot)|L_p(\R)\|\\
& \le\, C    \| \gf ^{-1} \psi_0\, |L_{u}(\R)\|\cdot   \|g|L_p(\R)\|.
\end{split}
\end{equation}
A similar argument yields the estimate of $\gf ^{-1}\psi_{0}\gf g $.
This proves \eqref{ct3}.\\
{\it Step 2.} To prove \eqref{ct4}, we put $h :=\gf ^{-1}\psi_j\gf f $. Then we have $h\in L_p^{\Gamma_j}(\R)$, hence  $h(2^{-j}\cdot) \in L_p^{\Gamma_0}(\R)$. In addition we know 
$\supp\varphi_{\bar{k}}(2^{j}\cdot )\subset \Gamma_0$. 
Let $\bar{k}=(k_1, \ldots\, k_d) $ such that $k_1, \ldots \, k_d \neq 0$. 
Using Lemma \ref{mul} we obtain
\begin{equation}\label{ws-15}
\begin{split}
\|\gf ^{-1}\varphi_{\bar{k}}\gf h|L_p(\R) \| &=\, 2^{-\frac{jd}{p}}\| (\gf ^{-1}\varphi_{\bar{k}}\gf h)(2^{-j}\cdot)|L_p(\R)\|\\
&=\, 2^{-\frac{jd}{p}}\| \gf ^{-1}\big[\varphi_{\bar{k}}(2^{j}\cdot) \gf [h(2^{-j}\cdot)]\big]|L_p(\R)\| \\
&\leq \, c\,  2^{-\frac{jd}{p}} \| \gf ^{-1}[\varphi_{\bar{k}}(2^{j}\cdot)]|L_{u}(\R)\|\cdot \|h(2^{-j}\cdot)|L_p(\R)\| \\
& \le  \, c\, \| \gf ^{-1}[\varphi_{\bar{k}}(2^{j}\cdot)]|L_{u}(\R)\|\cdot   \|h|L_p(\R)\|.
\end{split}
\end{equation}
We put $\bar{j}:= (j, \, \ldots \, , j)$. The homogeneity properties of the Fourier transform lead to
\begin{equation} 
\begin{split}\nonumber
\| \gf ^{-1}[\varphi_{\bar{k}}(2^{j}\cdot)]|L_{u}(\R)\| &=\, \| \gf ^{-1}[\varphi_{\bar{1}}(2^{-\bar{k}+\bar{j} + \bar{1}}\cdot)]|L_{u}(\R)\|\\
& \le \, c_1\, 2^{(jd-|\bar{k}|)(\frac{1}{u}-1)} \, \| \gf ^{-1}\varphi_{\bar{1}}|L_{u}(\R)\|.
\end{split}
\end{equation}
Inserting this into \eqref{ws-15} we get \eqref{ct4} for those $\bar{k}$. 
An obvious modification yields the estimate for the remaining $\bar{k}$.
The proof is complete.
\end{proof}


\subsection{Complex interpolation}


For the basics of  the classical complex interpolation method of Calder{\'o}n  we refer to the original paper \cite{Ca64} and 
the monographs \cite{BS,BL,lun,t78}. In the meanwhile it is well-known that the complex interpolation method extends to 
specific  quasi-Banach spaces, namely those, which are analytically convex, see \cite{kmm}. 
Note that any Banach space is analytically convex. 
The following Proposition, well-known in case of Banach spaces, see \cite[Theorem~4.1.2]{BL}, \cite[Theorem 2.1.6]{lun}
or \cite[Theorem~1.10.3.1]{t78}, can also be extended  to the quasi-Banach case, see \cite{kmm}.

\begin{proposition}\label{inter1}
Let $0 < \Theta < 1$.
Let $(X_1,Y_1)$ and $(X_2,Y_2)$ be two compatible couples of quasi-Banach spaces. In addition, let $X_1+X_2$, $Y_1+Y_2$ be analytically convex. 
If $T$ is in $\mathscr{L}(X_1,X_2)$ and in $\mathscr{L}(Y_1,Y_2)$, then the restriction of $T$ to $[X_1,Y_1]_{\Theta}$ is in 
$\mathscr{L}([X_1,Y_1]_{\Theta},[X_2,Y_2]_{\Theta})$ for every $\Theta$. Moreover, 
\beqq
\|T: [X_1,Y_1]_{\Theta}\to [X_2,Y_2]_{\Theta}\| \,\leq\, \|T:X_1\to X_2 \|^{1-\Theta}\, \|T:Y_1\to Y_2\|^{\Theta}.
\eeqq
\end{proposition}

It is not difficult to see that all spaces $B^{t}_{p,q}(\R)$ 
are analytically convex, see \cite{MM} or  \cite{kmm}.
By means of Theorem 7.8 in \cite{kmm} and the wavelet characterization of $S^{t}_{p,q} B(\R)$, see \cite{Vybiral}, one can derive 
that also the spaces  $S^{t}_{p,q} B(\R)$ are analytically convex.

\begin{proposition}\label{inter2}
Let $t_i\in \mathbb{R}$, $0<p_i,q_i\leq \infty$, $i=1,2$, and 
\[
\min \Big(\max(p_1,q_1), \max(p_2,q_2)\Big)<\infty \, .
\]
If $t_0,p_0$ and $q_0$ are given by
\be\label{ws-102}
\frac{1}{p_0}=\frac{1-\Theta}{p_1}+\frac{\Theta}{p_2},\qquad \frac{1}{q_0}=\frac{1-\Theta}{q_1}+\frac{\Theta}{q_2},\qquad t_0=(1-\Theta)t_1+\Theta t_2.
\ee
Then
\beqq
B^{t_0}_{p_0,q_0}(\R)=[B^{t_1}_{p_1,q_1}(\R),\, B^{t_2}_{p_2,q_2}(\R)]_{\Theta}
\eeqq
and
\beqq
S^{t_0}_{p_0,q_0}B(\R)=[S^{t_1}_{p_1,q_1}B(\R),\, S^{t_2}_{p_2,q_2} B(\R)]_{\Theta}.
\eeqq
\end{proposition}

\begin{remark}
 \rm 
Complex interpolation of isotropic Besov spaces has been  studied at various places, we refer to   \cite[Theorem~6.4.5]{BL} and 
\cite[2.4.1]{t78} as well as to the references given there. The extension to the quasi-Banach case has been done by Mendez and Mitrea \cite{MM}, 
see also  \cite{kmm}. 
Vybiral \cite[Theorem 4.6]{Vybiral} has proved a corresponding result for sequence spaces associated to Besov spaces  of dominating mixed smoothness. 
However, these results can be shifted to the level of function spaces by suitable  wavelet isomorphisms, see \cite[Theorem 2.12]{Vybiral}.
Here we would like to mention that all these extensions of the complex method to the quasi-Banach case in the particular situation of  Besov spaces in  
are based on investigations about corresponding Calder{\'o}n products, an idea, which goes back to the fundamental paper 
of Frazier and Jawerth \cite{fj90} (where Triebel-Lizorkin spaces are treated). 
\end{remark}

Later on we shall need also complex interpolation for some of the remaining cases not covered by Proposition \ref{inter2}. 
Let $X$ be quasi-Banach space  of distributions. By $\accentset{\diamond}{X}$ we denote the closure in $X$ of the set of all infinitely
differentiable functions $f$ such that
$D^\alpha f \in X$ for all $\alpha \in \N_0^d$.

\begin{proposition}\label{inter3}
Let $t_i\in \mathbb{R}$, $\Theta\in (0,1)$ and  $0<q_i\leq \infty$, $i=1,2$. 
If $t_0$ and $q_0$ are defined  as in \eqref{ws-102},
then
\be\label{ws-103}
\accentset{\diamond}{B}^{t_0}_{\infty,q_0}(\R)=[B^{t_1}_{\infty,q_1}(\R),\, B^{t_2}_{\infty,q_2}(\R)]_{\Theta}
\ee
and
\be\label{ws-104}
\accentset{\diamond}{S}^{t_0}_{\infty,q_0}B(\R)=[S^{t_1}_{\infty,q_1}B(\R),\, S^{t_2}_{\infty,q_2} B(\R)]_{\Theta},
\ee
\end{proposition}

\begin{proof}
The formula \eqref{ws-103} is proved in \cite{ysy}. 
Concerning \eqref{ws-104} one can argue in a similar way.
\end{proof}


\subsection{Dual spaces}


Next we will recall some results  about the dual spaces of $B^t_{p,q}(\R)$ and $S^t_{p,q}B(\R)$. 
For $1<p<\infty$ the conjugate exponent $p'$ is determined by $\frac{1}{p}+\frac{1}{p'}=1$.  
If $0<p\leq 1$ we put $p'=\infty$ and if $p=\infty$ we put $p'=1$.
For us it will be convenient to switch to the closure of $\cs(\R)$ in these  spaces.

\begin{definition}
 {\rm (i)} By $\mathring{B}^t_{p,q}(\R)$ we denote the closure of $\cs(\R)$ in $B^t_{p,q}(\R)$. \\
 {\rm (ii)} By $\mathring{S}^t_{p,q}B(\R)$ we denote the closure of $\cs(\R)$ in $S^t_{p,q}B(\R)$.  
\end{definition}

Recall that 
\[
\mathring{B}^t_{p,q}(\R)  = {B}^t_{p,q}(\R) \qquad \Longleftrightarrow \qquad \max (p,q) < \infty
\]
and 
\[
\mathring{S}^t_{p,q}B(\R)  = {S}^t_{p,q}B(\R) \qquad \Longleftrightarrow \qquad \max (p,q) < \infty\, .
\]
Because of the density of $\cs(\R)$ in these spaces any element of the dual space can be interpreted as an element of $\cs' (\R)$.
Hence, a distribution  $f\in \cs'(\R)$ belongs to the dual space $(\mathring{B}^t_{p,q}(\R))'$  
 if and only if there exists a positive constant $c$ such that
\beqq
|f(\varphi)|\leq c \, \|\varphi|B^t_{p,q}(\R)\|  \qquad \text{holds for all}\ \  \varphi \in \cs(\R).
\eeqq
Similarly for $\mathring{S}^t_{p,q}B(\R)$.

\begin{proposition}\label{dual1}
Let $t\in \re$.
\\
{\rm (i)} If $1\leq p < \infty$ and $0<q\le \infty$, then it holds
\beqq
[\mathring{B}^t_{p,q}(\R)]'=B^{-t}_{p',q'}(\R)\qquad \text{and}\qquad [\mathring{S}^t_{p,q}B(\R)]'=S^{-t}_{p',q'}B(\R).
\eeqq
{\rm (ii)} If $0<p<1$ and $0<q\le \infty$, then
\beqq
[\mathring{B}^t_{p,q}(\R)]'=B^{-t+d(\frac{1}{p}-1)}_{\infty,q'}(\R)\qquad \text{and}\qquad[\mathring{S}^t_{p,q}B(\R)]'=S^{-t+\frac{1}{p}-1}_{\infty,q'}B(\R)
\eeqq
\end{proposition}

\begin{proof}
The proof in the isotropic case can be found in \cite[Section 2.11]{Tr83}, see in particular Remark 2.11.2/2. 
For the  dominating mixed smoothness  we refer to \cite[Subsection 2.3.8]{Hansen}, at least if  $0<p,q<\infty$.
Here  we only  give a proof in case $q=\infty$ for the Besov spaces of 
dominating mixed smoothness following essentially the arguments given in \cite[2.5.1]{Tr78} for the isotropic case.
\\
{\it Step 1.} We shall prove $[\mathring{S}^t_{p,\infty}B(\R)]'=S^{-t}_{p',1}B(\R)$. \\
{\it Substep 1.1.}  Let $(\varphi_{j})_{j\in \N_0}$ be the univariate smooth dyadic decomposition of unity used in the definition of the spaces. 
We put 
\[\tilde{\varphi}_{j}:=\varphi_{j-1}+\varphi_j+\varphi_{j+1}\, , \qquad j=0,1, \ldots \, , 
   \]
 with $\varphi_{-1}\equiv 0$. For $\bar{k}\in \N_0^d$ we define $\tilde{\varphi}_{\bar{k}}:=\tilde{\varphi}_{k_1} \, \otimes \, \cdots \, \otimes \, \tilde{\varphi}_{k_d}$. 
With  $f\in S^{-t}_{p',1}B(\R)$ and  $\psi\in S(\R)$ we have
\begin{equation}\nonumber
\begin{split}
|f(\psi)|&=\,\Big|\sum_{\bar{k}\in \N_0^d}(\gf^{-1}\varphi_{\bar{k}}\gf f)(\psi)\Big|=\Big|\sum_{\bar{k}\in \N_0^d}(\gf^{-1}\tilde{\varphi}_{\bar{k}}\gf \gf^{-1}\varphi_{\bar{k}}\gf f)(\psi)\Big|\\
&=\,\Big|\sum_{\bar{k}\in \N_0^d}( \gf^{-1}\varphi_{\bar{k}}\gf f)(\gf^{-1}\tilde{\varphi}_{\bar{k}}\gf\psi)\Big|\\
&\leq \,\|2^{-|\bar{k}|t} \gf^{-1}\varphi_{\bar{k}}\gf f|\ell_1(L_{p'})\|\cdot \|2^{|\bar{k}|t} \gf^{-1}\tilde{\varphi}_{\bar{k}}\gf\psi|\ell_{\infty}(L_p)|.
\end{split}
\end{equation}
Observe that
\begin{equation}\nonumber
\begin{split}
 \|2^{|\bar{k}|t} \gf^{-1}\tilde{\varphi}_{\bar{k}}\gf\psi|\ell_{\infty}(L_p)\| 
 &\leq\, \sum_{ \| \bar{\ell}\|_\infty \le 1} \| 2^{|\bar{k}|t}\gf \varphi_{\bar{k}+\bar{\ell}}\gf^{-1}\psi |\ell_{\infty}(L_p)\|\\
 &\leq\,  c\, \| \psi | S^{t}_{p,\infty}B(\R)\|
\end{split}
\end{equation}
for some $c>0$ independent of $\psi$ and $f$. Consequently 
\beqq
|f(\psi)|\leq c\,  \| f|S^{-t}_{p',1}B(\R)\| \cdot \| \psi| S^{t}_{p,\infty}B(\R)\|
\eeqq
which means $ f\in [\mathring{S}^{t}_{p,\infty}B(\R)]'$.
\\
{\it Substep 1.2.} Next we prove the reverse direction. 
We assume that the generator of our smooth dyadic decomposition of unity is an even function.
Then
$\varphi_{\bar{k}} (-x) = \varphi_{\bar{k}} (x)$ follows for all $x\in \R$ an all $\bar{k} \in \N_0^d$.
\\
Let $c_0(L_p)$ denote the space of all sequences $(\psi_{\bar{k}})_{\bar{k}}$ of measurable functions such that
\[
\lim_{|\bar{k}| \to \infty} \, \|\, \psi_{\bar{k}}\, |L_p(\R) \| = 0
\]
equipped with the norm
\[
 \| \, (\psi_{\bar{k}})_{\bar{k}}\, |c_0(L_p)\| := \sup_{\bar{k} \in \N_0^d}\, \|\, \psi_{\bar{k}}\, |L_p(\R) \|\, .
\]
Observe that  
\[
 J:  ~ g \mapsto (2^{|\bar{k}|t}\gf^{-1}\varphi_{\bar{k}}\gf g)_{\bar{k}\in \N_0^d}
\]
 is  isometric and bijective if $J$ is considered as a mapping 
from $\mathring{S}^t_{p,\infty}B(\R)$ onto a closed subspace $Y$ of $c_0(L_p)$. 
Here we use the fact that 
\[
\lim_{|\bar{k}| \to \infty} \, \|\, 2^{|\bar{k}|t}\, \gf^{-1}\varphi_{\bar{k}}\gf g \, |L_p(\R) \| = 0
\] 
holds for all $g \in \mathring{S}^{t}_{p,\infty}B(\R)$.
\\
Let $f\in [\mathring{S}^{t}_{p,\infty}B(\R)]'$.
Hence, by defining
\[
\tilde{f} \Big((\psi_{\bar{k}})_{\bar{k}}\Big):= f\Big(\sum_{\bar{k} \in \N_0^d} 2^{-|\bar{k}|t}\, \psi_{\bar{k}}\Big) \, , \qquad (\psi_{\bar{k}})_{\bar{k}} \in Y\, , 
\]
$\tilde{f}$ becomes  a linear and continuous  functional on $Y$ satisfying $\| \, \tilde{f} \, | Y \to \C\| = \| \, f \, |\, [\mathring{S}^{t}_{p,\infty}B(\R)]'\| $. 
Now, by the Hahn-Banach theorem, there exists a linear and continuous  extension of $\tilde{f}$ 
to a continuous linear functional on the space $c_0(L_p)$. 
It is  known that $[c_0(L_p)]'=\ell_1(L_{p'})$ and any $g \in [c_0(L_p)]'$ can be represented
in the form
\be\label{dua-1}
g ((\psi_{\bar{k}})_{\bar{k}}) = \sum_{\bar{k}\in \N_0^d}\int_{ \R} g_{\bar{k}}(x) \, \psi_{\bar{k}}(x)\, dx\, , 
\qquad (\psi_{\bar{k}})_{\bar{k}}
\in c_0(L_p)\, ,
\ee
where the functions $g_{\bar{k}}$ satisfy
\[
\| \, g \, | \ell_1 (L_{p'})\| = \sum_{\bar{k}\in \N_0^d} \| g_{\bar{k}}\, |L_{p'}(\R)\| < \infty \, ,
\]
see \cite[Lemma 1.11.1]{t78}.
Applying this with $g= \tilde{f}$ we find
\be\label{dua-2}
\| \, (f_{\bar{k}})_{\bar{k}}\, |\ell_1(L_{p'})\| = 
\| \, \tilde{f} \, | [c_0(L_p)]'\| = \| \, f \, |\, [\mathring{S}^{t}_{p,\infty}B(\R)]'\| 
\ee 
for an appropriate sequence $ (f_{\bar{k}})_{\bar{k}}$.
In view of \eqref{dua-1}, the definition of $\tilde{f}$, the Plancherel identity and the symmetry condition with respect to $(\varphi_{\bar{k}})$ we obtain 
\begin{equation}
\begin{split}\nonumber
f(\psi)  = &\, f \Big(\sum_{\bar{k}\in \N_0^d} \,  \gf^{-1}\varphi_{\bar{k}}\gf \psi \Big) =  \tilde{f} \Big((2^{|\bar{k}|t}\, \gf^{-1}\varphi_{\bar{k}}\gf \psi)_{\bar{k}} \Big)
\\
=&\,
\sum_{\bar{k}\in \N_0^d} 2^{|\bar{k}|t}\, \int f_{\bar{k}}(x)\, (\gf^{-1}\varphi_{\bar{k}}\gf  \psi) (x)\, dx\\
 =&\,
\sum_{\bar{k}\in \N_0^d} 2^{|\bar{k}|t}\, \int \psi (x)\, (\gf^{-1}\varphi_{\bar{k}}\gf   f_{\bar{k}}) (x)\, dx
\\
=&\, \sum_{\bar{k}\in \N_0^d} \, 2^{|\bar{k}|t}\, \Big(\gf^{-1}\varphi_{\bar{k}}\gf   f_{\bar{k}}\Big) (\psi)
\end{split}
\end{equation}
for any $\psi \in \cs (\R)$. 
This leads to the identity 
\beqq
\varphi_{\bar{k}}\gf f= \sum_{{\| \, \bar{\ell}\, \|_\infty \le 1 \atop  
\ell_j+k_j\geq 0\, ,  \ j=1, \ldots\, ,d}} \, 2^{|\bar{k}+\bar{\ell}|t}\, \varphi_{k} \, \varphi_{\bar{\ell}+\bar{k}}\, \gf f_{\bar{k}+\bar{\ell}} \,,
\eeqq
valid in $\cs'(\R)$. Consequently, by using a standard convolution inequality and a homogeneity argument, we have
\begin{equation}
\begin{split}\nonumber
\| \, \gf^{-1}\varphi_{\bar{k}}&\gf f\, |L_{p'}(\R)\|\\
& \leq\, \sum_{{\| \, \bar{\ell}\, \|_\infty \le 1 \atop  
\ell_j+k_j\geq 0\, ,  \ j=1, \ldots\, ,d}} 2^{|\bar{k}+\bar{\ell}|t}\, 
\| \, \gf^{-1}\varphi_{k}\, \varphi_{\bar{\ell}+\bar{k}}\, \gf f_{\bar{k}+\bar{\ell}}\, |L_{p'}(\R) \|\\
& \leq\,
\sum_{{\| \, \bar{\ell}\, \|_\infty \le 1 \atop  
\ell_j+k_j\geq 0\, ,  \ j=1, \ldots\, ,d}}\, 2^{|\bar{k}+\bar{\ell}|t}\, 
\| \, \gf^{-1} [\varphi_{k}\, \varphi_{\bar{\ell}+\bar{k}}]\, \, |L_{1}(\R) \|
\|f_{\bar{k}+\bar{\ell}}|L_{p'}(\R) \|
\\
&\leq\, c_1\, 
\sum_{{\| \, \bar{\ell}\, \|_\infty \le 1 \atop  
\ell_j+k_j\geq 0\, ,  \ j=1, \ldots\, ,d}}\, 2^{|\bar{k}+\bar{\ell}|t}\, 
\|f_{\bar{k}+\bar{\ell}}|L_{p'}(\R) \|
\end{split}
\end{equation}
with $c_1$ independent of $f$ and $\bar{k}$. Therefore
\beqq
\| f|S^{-t}_{p',1}B(\R)\| \,\leq\, c_2\, \|\, (f_{\bar{k}})_{\bar{k}}\, |\ell_1(L_{p'})\|
\eeqq
follows.
This together with \eqref{dua-2} proves that 
\beqq
\| f|S^{-t}_{p',1}B(\R)\| \leq c_2 \, \|f| [\mathring{S}^{t}_{p,\infty}B(\R)]'\| 
\eeqq 
holds with a constant $c_2$ independent of $f$. 
\\
{\it Step 2. } We shall prove $[\mathring{S}^t_{p,\infty}B(\R)]'=S^{-t+\frac{1}{p}-1}_{\infty,1}B(\R)$. \\
{\it Substep 2.1.}  The embedding $S^t_{p,\infty}B(\R)\hookrightarrow S^{t-\frac{1}{p}+1}_{1,\infty}B(\R)$
implies
\beqq
\mathring{S}^t_{p,\infty}B(\R)\hookrightarrow \mathring{S}^{t-\frac{1}{p}+1}_{1,\infty}B(\R).
\eeqq
Duality and Step 1 yields
\beqq
S^{-t+\frac{1}{p}-1}_{\infty,1}B(\R)\hookrightarrow [\mathring{S}^t_{p,\infty}B(\R)]'\,.
\eeqq
{\it Substep 2.2.} Let $f \in[\mathring{S}^t_{p,\infty}B(\R)]' $. Following Hansen \cite[page 75]{Hansen} we choose a point $x_{\bar{k}}\in \R$
for any $\bar{k} \in \N_0^d$  such that
\be\label{dua-3}
\frac{1}{2}\| \gf^{-1}\varphi_{\bar{k}}\gf f|L_{\infty}(\R)\| \leq |(\gf^{-1}\varphi_{\bar{k}}\gf f)(x_{\bar{k}})|\leq 
\| \gf^{-1}\varphi_{\bar{k}}\gf f|L_{\infty}(\R)\|\, .
\ee
Then we define the function 
\beqq
\psi(x):=\sum_{|\bar{\ell}|\leq n} \, a_{\bar{\ell}}(\gf^{-1}\varphi_{\bar{\ell}})(x_{\bar{\ell}}-x)\, 
2^{|\bar{\ell}|(-t+\frac{1}{p}-1)}\, , \qquad x \in \R\, .
\eeqq
Obviously $\psi\in \cs(\R)$. An easy calculation yields
\begin{equation}
\begin{split}\nonumber
\|\, \psi\, &| S^t_{p,\infty}B(\R)\| \\
&=\, \sup_{\bar{k}\in \N_0^d} \, 2^{|\bar{k}|t}
\bigg\| \cfi \Big(\sum_{\|\, \bar{\ell}\, \|_\infty \le 1\atop  \,|\bar{k}+\bar{\ell}|\leq n} \, 2^{|\bar{k}+\bar{\ell}|(-t+\frac{1}{p}-1)}
\\
&\,\qquad\qquad\qquad \times  
\, a_{\bar{k}+\bar{\ell}}\, \varphi_{\bar{k}}(\xi) \, \varphi_{\bar{k}+\bar{\ell}}(-\xi) \, e^{-ix_{(\bar{k}+\bar{\ell})}\xi}\Big) (\, \cdot \, )
\, \bigg|L_p(\R)\bigg\|\\
& \le \,\sup_{\bar{k}\in \N_0^d} \, 
2^{|\bar{k}|t} \sum_{\|\, \bar{\ell}\, \|_\infty \le 1\atop |\bar{k}+\bar{\ell}|\leq n}
\, \big\|\, 2^{|\bar{k}+\bar{\ell}|(-t+\frac{1}{p}-1)}\, a_{\bar{k}+\bar{\ell}}\, 
\gf^{-1}[\varphi_{\bar{k}}\varphi_{\bar{k}+\bar{\ell}}(-\cdot)](\, \cdot \, )\, |L_p(\R)\big\|\\
&\leq \, c_1\,  \sup_{\bar{k}\in \N_0^d}\, 2^{|\bar{k}|t} 
\sum_{\| \, \bar{\ell}\, \|_\infty \le 1\atop |\bar{k}+\bar{\ell}|\leq n}
\, \big\|\, 2^{|\bar{k}+\bar{\ell}|(-t+\frac{1}{p}-1)}\, a_{\bar{k}+\bar{\ell}}\, 
\gf^{-1}[\varphi_{\bar{k}+\bar{\ell}}(-\cdot)] (\, \cdot \, )\, |L_p(\R)\big\|
\end{split}
\end{equation}
where the last inequality is a consequence of Lemma \ref{mul} and a homogeneity argument. 
Observe that 
\beqq
\|\gf^{-1}[\varphi_{\bar{k}+\bar{\ell}}(-\cdot)] (\, \cdot \, )\, |L_p(\R)\|=\|\, \gf \varphi_{\bar{k}+\bar{\ell}}\,  |L_p(\R)\|=
2^{|\bar{k}+\bar{\ell}|(1-\frac{1}{p}) }\, \|\, \gf \varphi_{\bar{1}}\,  |L_p(\R)\|
\eeqq
if $\bar{k}_i +\bar{\ell}_i\le 1$ for all $i=1, \ldots \, , d$.
If $\min (\bar{k}_i +\bar{\ell}_i)=0$ one has to modify this in an obvious way.
Altogether we have found
\beqq
\|\, \psi\, | S^t_{p,\infty}B(\R)\| \leq c_2 \, \sup_{\bar{k}\in \N_0^d}  \sum_{\|\, \bar{\ell}\, \|_\infty \le 1\, , 
\,|\bar{k}+\bar{\ell}|\leq n} \, |a_{\bar{k}+\bar{\ell}}| \leq c_3 \, \sup_{|\bar{k}|\leq 2\, n}\, |a_{\bar{k}}|\,.
\eeqq
This estimate can be used to derive 
\begin{equation}
\begin{split}\nonumber
\bigg|\sum_{|\bar{k}|\leq n}\, a_{\bar{k}}\, 2^{|\bar{k}|(-t+\frac{1}{d}-1)}(\gf^{-1}\varphi_{\bar{k}}\gf f)(x_{\bar{k}})\bigg|
&=\,\bigg|\sum_{|\bar{k}|\leq n}a_{\bar{k}}2^{|\bar{k}|(-t+\frac{1}{d}-1)}(f*\gf^{-1}\varphi_{\bar{k}})(x_{\bar{k}})\bigg|\\
&=\, |f(\psi)|\\
&\leq \, \|f\| \cdot \|\psi|S^t_{p,\infty}B(\R)\|\\
&\leq \, c_3\, \| f \|\,  \sup_{|\bar{k}|\leq 2n}\, |a_{\bar{k}}|\,.
\end{split}
\end{equation}
Employing  \eqref{dua-3} and the fact that the $a_{\bar{k}}$ can be chosen as we want, 
for instance such that
\[
a_{\bar{k}}\, (\gf^{-1}\varphi_{\bar{k}}\gf f)(x_{\bar{k}}) = 
|\gf^{-1}\varphi_{\bar{k}}\gf f)(x_{\bar{k}})|\, , 
\]
we  find
\beqq
 \sum_{|\bar{k}|\leq n} \, 2^{|\bar{k}|(-t+\frac{1}{p}-1)}\, \|\, \gf^{-1}\varphi_{\bar{k}}\gf f\, |L_{\infty}(\R)\|\leq c_3\, \| f\|\,.
 \eeqq
Here $c_3$ is independent of $f$ and $n$.
For $n\to \infty$ we obtain
 \beqq
 \|\, f\, |S^{-t+\frac{1}{p}-1}_{\infty,1} B(\R)\|\leq c_3\, \|f\|\,.
\eeqq
The proof is complete.
\end{proof}


\subsection{Test functions}


Let $d \ge 2$.
Before we are going to define some test functions we mention a few more properties of our smooth decompositions of unity.
As a consequence of the definitions we obtain
\[
\varphi_{\bar{k}} (x) = 1 \qquad \mbox{if}\qquad \frac 34 \, 2^{k_j} \le x_j\le 2^{k_j}\, , \qquad j=1, \ldots \, d\, ,  
\]
if $\min_{j=1, \ldots \, ,d} k_j >0$.
In case $\min_{j=1, \ldots \, ,d} k_j =0$ the following statement is true:
$\varphi_{\bar{k}} (x) = 1$  if $\frac 34 \, 2^{k_j} \le x_j\le 2^{k_j}$ holds for all $j$ such that $k_j =0$
and $ 0 \le x_j\le 1$ for the remaining components. 
\\
Now we switch to $(\psi_\ell)_{\ell}$. For $\ell \in \N$ it follows
\[
\psi_\ell (x) = 1 \quad \mbox{on the set} \quad \{x:\ \sup_{j=1,...,d} |x_j| \le 2^{\ell}\}\setminus 
\{x: \ \sup_{j=1,...,d} |x_j| \le \frac 34 \, 2^{\ell}\}\, .
\]
 
\subsection*{Example 1}

Let $g\in C_0^\infty(\R)$ be a function such that  $\supp g\in B(0,\epsilon)$ for $\epsilon >0$ small enough $(0<\epsilon<\frac{1}{8})$
and $|\cfi g (\xi)| >0$ on $[\pi,\pi]^d$.
For $\ell \in  \N$ we define the family of functions $f_\ell$ by
\be\label{ws-example1}
\gf f_\ell  (\xi) := \sum_{j=1}^{\ell} a_j\,  g(\xi_1-\frac{7}{8}2^{\ell},\xi_2-\frac{7}{8}2^j,\xi_3,...,\xi_d)\, , \qquad \xi \in \R \, , 
\ee
where the sequence $(a_j)_{j=1}^\ell$ of complex numbers will be chosen later on.
Note that
\begin{equation}
\begin{split}\nonumber
\supp g(\cdot-\frac{7}{8}2^{\ell},\cdot-\frac{7}{8}2^j,\cdot,...,\cdot) 
& \subset\, \{x: \ \varphi_{\bar{k}}(x)=1,\, \quad \bar{k}=(\ell,j,0,...,0)\}\\
&\subset \, \{x: \ \psi_{\ell}(x)=1 \}\, , \quad j=1, \ldots \, , \ell\, .
\end{split}
\end{equation}
It follows
\[
\gf^{-1}[\psi_{m} \gf f_\ell ]  = \delta_{m,\ell} \, f_\ell
\]
and 
\[
\gf^{-1}[\varphi_{\bar{k}} \gf f_\ell ]  = \delta_{\bar{k}, (\ell,j,0,...,0)}\, 
a_j\, \gf^{-1}  [g(\xi_1-\frac{7}{8}2^{\ell},\xi_2-\frac{7}{8}2^j,\xi_3,...,\xi_d)]\, .
\]
Hence
\begin{equation}
\begin{split}\label{ws-08}
 \|\, f_\ell\, |B^{t}_{p,q}(\R)\|  & = \,      2^{\ell t}    \, \big\|\gf^{-1}[\psi_{\ell}\gf f](\cdot)|L_p(\R)\big\|
\\
& = \, 2^{\ell t} \, \Big\|\sum_{j=1}^{\ell} a_j \, \gf^{-1} [g(\xi_1-\frac{7}{8}2^{\ell},\xi_2-\frac{7}{8}2^j,\xi_3,...,\xi_d)]\, 
\Big|L_p(\R)\Big\|
\\
& = \, 2^{\ell t} \, \Big\|\gf^{-1} g (x)\, \Big(\sum_{j=1}^{\ell} a_j \, e^{\frac{7}{8}i(2^\ell x_1 + 2^j x_2)} \Big)\,  \Big|L_p(\R)\Big\| 
\\
& \asymp \, 2^{\ell t} \, \Big\|\, \sum_{j=1}^{\ell} a_j \, e^{\frac{7}{8}i(2^\ell x_1 + 2^j x_2)} \,  \Big|L_p([-\pi,\pi]^2)\Big\| 
\, .
\end{split}
\end{equation}
For the last step we used that $\gf^{-1} g $ is rapidly decreasing, $|\cfi g (\xi)| >0$ on $[\pi,\pi]^d$ and that $\R$ can be written as 
\[
\R = \bigcup_{m \in \Z} [2m\pi, 2(m+1)\pi) \, .
\]
In case $1 <p<\infty$ a Littlewood-Paley characterization of $L_p([-\pi,\pi]^2)$ yields
\be\label{ws-09}
 \|\, f_\ell\, |B^{t}_{p,q}(\R)\|\asymp 2^{\ell t} \, \Big(\sum_{j=1}^{\ell} |a_j|^2\Big)^{1/2} \,.
\ee
Similarly
\begin{equation}
\begin{split}\label{ws-10}
\|\, f_\ell\, |S^{t}_{p,q}(\R)\|  = &    \Big(\sum_{j=1}^\ell  2^{jtq}    \, \big\|\gf^{-1}[\varphi_{(\ell,j,0,\ldots \, ,0)} \gf f_\ell](\cdot)
\, |L_p(\R)\big\|^q\Big)^{1/q}
\\
 = & \Big(\sum_{j=1}^\ell  2^{jtq} \, |a_j|^q   \, \big\|  \, \gf^{-1}  [g(\xi_1-\frac{7}{8}2^{\ell},\xi_2-\frac{7}{8}2^j,...,\xi_d)]
\, |L_p(\R)\big\|^q\Big)^{1/q}
\\
= & \big\|  \, \gf^{-1}  g\, |L_p(\R)\big\| \, \Big(\sum_{j=1}^\ell  2^{jtq} \, |a_j|^q  \Big)^{1/q}
\, .
\end{split}
\end{equation}

\subsection*{Example 2}

In case $p=\infty$ nontrivial periodic functions are contained in $B^t_{\infty,q}(\R)$ and 
$S^t_{\infty,q}B(\R)$.
So we can work directly with lacunary series.
Let 
\be\label{ws-example2}
 f_\ell  (x) := \sum_{j=1}^{\ell} a_j\,  e^{i ( 2^{\ell}x_1 + 2^j x_2)}\, , \qquad x=(x_1, \ldots \, ,x_d) \in \R \, .
\ee
Then
\[
\gf^{-1}[\psi_{m} \gf f_\ell ]  = \delta_{m,\ell} \, f_\ell
\]
and 
\[
\gf^{-1}[\varphi_{\bar{k}} \gf f_\ell ]  = \delta_{\bar{k}, (\ell,j,0,...,0)}\, 
a_j\, e^{i ( 2^{\ell}x_1 + 2^j x_2)} 
\]
follow. For $a_j \ge 0$ for all $j$ this will allow us to calculate the quasi-norms in  $B^t_{\infty,q}(\R)$ and 
$S^t_{\infty,q}B(\R)$.
We obtain in the first case
\begin{equation}
\begin{split}
\label{ws-11}
 \|\, f_\ell\, |B^{t}_{\infty,q}(\R)\|   = &  \,    2^{\ell t} \, \big\|\gf^{-1}[\psi_{\ell}\gf f](\cdot)|L_\infty(\R)\big\|
\\
 = &\, 2^{\ell t} \, \sup_{x \in \R} \, \Big|\sum_{j=1}^{\ell} a_j \,  e^{i ( 2^{\ell}x_1 + 2^j x_2)}\,\Big|  
\\
 = &\, 2^{\ell t} \, \sum_{j=1}^{\ell} a_j \, .
\end{split}
\end{equation}
Concerning the dominating mixed smoothness we conclude
\begin{equation}
\begin{split}
\label{ws-12}
 \|\, f_\ell\, |S^{t}_{\infty,q}(\R)\|   = & \,   \Big(\sum_{j=1}^\ell  2^{jtq}    \, 
\big\|\gf^{-1}[\varphi_{(\ell,j,0,\ldots \, ,0)} \gf f_\ell](\cdot) \, |L_\infty(\R)\big\|^q\Big)^{1/q}
\\
= &  \, \Big(\sum_{j=1}^\ell  2^{jtq} \, |a_j|^q  \Big)^{1/q}
\, .
\end{split}
\end{equation}
\subsection*{Example 3}

Let us consider a function   $g\in C_0^\infty(\re)$ such that 
$\supp g\subset \{x\in \re: 3/2\leq |x|\leq 2\} $. For $j\in \N$, $\bar{k}\in \N^d$ we define
\beqq
g_j(t)=g(2^{-j+1}t)
\quad\text{and}\quad
g_{\bar{k}}(x)=g_{k_1}(x_1)\cdots g_{k_d}(x_d)\, , \quad t \in \re,\ x \in \R\, .
\eeqq
Let  
\[
\nabla_{\ell}:=\{\bar{k}\in \N^d,\ |\bar{k}|_{\infty}=\ell\}\, , \qquad \ell \in \N\, .
 \]
 Then, if $\bar{k}\in \nabla_{\ell}$, we have
\beqq
\supp g_{\bar{k}}
\subset \{x: \ \varphi_{\bar{k}}(x)=1\} \subset  \{x: \ \psi_{\ell}(x)=1 \}.
\eeqq
We define the family of test functions 
\beqq
f_{\ell}=\sum_{\bar{k}\in \nabla_{\ell}}\, a_{\bar{k}}\, \gf^{-1} g_{\bar{k}}\, , \qquad \ell \in \N^d\, . 
\eeqq
The coefficients $(a_{\bar{k}})_{\bar{k}}$ will be chosen later on.
By construction we have 
\begin{equation}
\begin{split}\nonumber
\|f_{\ell}|S^{0}_{p,q}B(\R)\| & = \,  \bigg(\sum\limits_{\bar{k}\in \N_0^d}
              \big\|\gf^{-1}[\varphi_{\bar{k}}\gf f_\ell](\cdot)|L_p(\R)
             \big\|^q\bigg)^{1/q}\nonumber\\
 & = \,  \bigg(\sum_{\bar{k}\in \nabla_{\ell}} |a_{\bar{k}}|^q\big\|  \gf^{-1}g_{\bar{k}}\big|L_p(\R) \big\|^q\bigg)^{1/q}\, .
\end{split}
\end{equation}
Observe that
\beqq
\big\|  \gf^{-1}g_{\bar{k}}\big|L_p(\R) \big\| =2^{(|\bar{k}|-d)(1-\frac{1}{p})}\, 
\big\|\,   \gf^{-1}g_{\bar{1}}\big|L_p(\R) \big\|= C \, 2^{|\bar{k}|(1-\frac{1}{p})}
\eeqq
for an appropriate $C>0$ (independent of $\ell$).
Consequently we obtain
 \beqq
 \|f_{\ell}|S^{0}_{p,q}B(\R)\| \, = \,   C\, \bigg(\sum_{\bar{k}\in \nabla_{\ell}} |a_{\bar{k}}|^q\, 2^{|\bar{k}|(1-\frac{1}{p})q}\bigg)^{1/q}\,, \qquad 
\ell \in \N \, .
 \eeqq
 Next we compute
\begin{equation}
\begin{split}\nonumber
 \|f_{\ell}|B^{0}_{p,q}(\R)\| & = \,
            \bigg(\sum\limits_{j=0}^{\infty}
            \big\|\gf^{-1}[\psi_{j}\gf f_\ell](\cdot)|L_p(\R)
            \big\|^q\bigg)^{1/q}\\ 
 & = \,      \big\|\sum_{\bar{k}\in \nabla_{\ell}}a_{\bar{k}}\gf^{-1}g_{\bar{k}}|L_p(\R)  \Big\|.
\end{split}
\end{equation}
Recall, for $0<p_0<p<p_1<\infty$ we have
\beqq
S^{\frac{1}{p_0}-\frac{1}{p}}_{p_0,p}B(\R)\hookrightarrow S^0_{p,2}F(\R)\hookrightarrow S^{\frac{1}{p_1}-\frac{1}{p}}_{p_1,p}B(\R)\, ,
\eeqq
see \cite{HV}, and  $S^0_{p,2}F(\R)=L_p(\R)$,  $1<p<\infty$, see \cite{Li}. These arguments    lead to
\begin{equation}
\begin{split}\nonumber
\big\|\sum_{\bar{k}\in \nabla_{\ell}}a_{\bar{k}}\gf^{-1}g_{\bar{k}}|L_p(\R)  \Big\| 
\leq&\, C_1 \Big(\sum_{\bar{k}\in \nabla_{\ell}}  2^{|\bar{k}|(\frac{1}{p_0}-\frac{1}{p})p} |a_{\bar{k}}|^p\|\gf^{-1}g_{\bar{k}}|L_{p_0}(\R)\|^p\Big)^{\frac{1}{p}}  \\
=& \, C_2 \Big(\sum_{\bar{k}\in \nabla_{\ell}}  2^{|\bar{k}|(\frac{1}{p_0}-\frac{1}{p})p} 
|a_{\bar{k}}|^p 2^{|\bar{k}|(1-\frac{1}{p_0})p}\Big)^{\frac{1}{p}}\\
=&\, C_2 \Big(\sum_{\bar{k}\in \nabla_{\ell}}   |a_{\bar{k}}|^p2^{|\bar{k}|(1-\frac{1}{p})p}\Big)^{\frac{1}{p}}\,.
\end{split}
\end{equation}
Similarly we have
\begin{equation}
\begin{split}\nonumber
\Big(\sum_{\bar{k}\in \nabla_{\ell}}  2^{|\bar{k}|(\frac{1}{p_1}-\frac{1}{p})p}
 |a_{\bar{k}}|^p\|\gf^{-1}g_{\bar{k}}|L_{p_1}(\R)\|^p\Big)^{\frac{1}{p}}
=&\, C_3 \Big(\sum_{\bar{k}\in \nabla_{\ell}}   |a_{\bar{k}}|^p2^{|\bar{k}|(1-\frac{1}{p})p}\Big)^{\frac{1}{p}}\\
\leq &\, C_4 \big\|\sum_{\bar{k}\in \nabla_{\ell}}a_{\bar{k}}\gf^{-1}g_{\bar{k}}|L_p(\R)\big\|\,.
\end{split}
\end{equation}
Altogether we have proved in case $1 <p< \infty $
\beqq
 \|f_{\ell}|B^{0}_{p,q}(\R)\| \, \asymp \,  \Big(\sum_{\bar{k}\in \nabla_{\ell}}   |a_{\bar{k}}|^p2^{|\bar{k}|(1-\frac{1}{p})p}\Big)^{\frac{1}{p}}\,,
\eeqq
where the positive constants behind $\asymp$ do not depend on $\ell \in \N$.


\subsection*{Example 4}


We consider the same basic functions $g_{\bar{k}}$ as in Example 3. This time we define  
\beqq
f_{\ell} := \sum_{j=1}^{\ell } \, a_{j}\, \gf^{-1}g_{\bar{j}}\,,\qquad \bar{j}:=(j,1,\cdots,1)\,.
\eeqq
As above we conclude
\begin{equation}
\begin{split}\nonumber
\|f_{\ell}|S^{t}_{p,q}B(\R)\| 
 & = \,  \bigg(\sum_{j=1}^{\ell} 2^{tjq}|a_{j}|^q\big\|  \gf^{-1}g_{\bar{j}_0}\big|L_p(\R) \big\|^q\bigg)^{1/q}\nonumber\\
 & = \, C \bigg(\sum_{j=1}^{\ell} 2^{j(t+1-\frac{1}{p})q}|a_{j}|^q\bigg)^{1/q}
\end{split}
\end{equation}
and
\beqq
 \|f_{\ell}|B^{t}_{p,q}(\R)\|  
 \, = \, C \bigg(\sum_{j=1}^{\ell} 2^{j(t+1-\frac{1}{p})q}|a_{j}|^q\bigg)^{1/q}
\eeqq
for an appropriate positive constant $C$ (independent of $\ell$).
Notice that we do not need the restriction $1 <p< \infty$ here.
It is true for all $p$.


\subsection*{Example 5}


This will be one more modification of Example 3.
Let $g_{\bar{k}}$ be defined  as there. We put 
\beqq
f_{\ell} : = \sum_{j=1}^{\ell }\, a_{j}\, \gf^{-1}g_{\bar{j}}\,\qquad \bar{j}=(j,\cdots,j)\,.
\eeqq
Then we have 
\begin{equation}
\begin{split}
 \|f_{\ell}|S^{t}_{p,q}B(\R)\| 
 = &\,  \bigg(\sum_{j=1}^{\ell} 2^{tdjq}|a_{j}|^q\big\|  \gf^{-1}g_{\bar{j}}\big|L_p(\R) \big\|^q\bigg)^{1/q}\nonumber\\
  = &\, C  \bigg(\sum_{j=1}^{\ell} 2^{jd(t+1-\frac{1}{p})q}|a_{j}|^q\bigg)^{1/q}
\end{split}
\end{equation}
and
\beqq
 \|f_{\ell}|B^{t}_{p,q}(\R)\| 
 \, = \, C  \bigg(\sum_{j=1}^{\ell} 2^{jd(\frac{t}{d}+1-\frac{1}{p})q}|a_{j}|^q\bigg)^{1/q}\,,
\eeqq
where $C$ is as in Example 4. Also here the restriction $1 <p < \infty$ is not needed.


\subsection*{Example 6}


This example is taken from \cite[2.3.9]{Tr83}.
Let $\varrho \in \cs (\R) $ be a function such that 
$\supp \cf \varrho \subset \{\xi : \: |\xi|\le 1 \} $.
We define
\[
h_j(x):= \varrho (2^{-j} x)\, , \qquad x \in \R\, , \quad j \in \N\, .
\]
For all $p,q,t$ we conclude
\beqq
\|\, h_j \, |B^t_{p,q}(\R)\| = \|\, h_j \, |L_{p}(\R)\| = 2^{jd/p}\, \|\, \varrho \, |L_{p}(\R)\|\, , \qquad j \in \N\, .   
\eeqq
Similarly, also for all $p,q,t$,  we obtain
\beqq
\|\, h_j \, |S^t_{p,q}B(\R)\| = \|\, h_j \, |L_{p}(\R)\| = 2^{jd/p}\, \|\, \varrho \, |L_{p}(\R)\|\, ,\qquad j \in \N\, .      
\eeqq
As an immediate consequence of these two identities we get the following.

\begin{lemma}\label{p0p1}
Let $0 < p_0,p_1,q_0,q_1 \le \infty$ and $t_0,t_1 \in \re$.
\\
{\rm (i)} An embedding $S^{t_0}_{p_0,q_0}B(\R) \hookrightarrow B^{t_1}_{p_1,q_1}(\R)$ implies $p_0 \le p_1$.
\\
{\rm (ii)} An embedding $B^{t_0}_{p_0,q_0}(\R) \hookrightarrow S^{t_1}_{p_1,q_1}B(\R)$ implies $p_0 \le p_1$.
\end{lemma}


\subsection{Proof of Theorem \ref{besov4} -- sufficiency}
\label{proof2}


{\em Step 1.} Preparations.
For  $\bar{k}\in \N_0^d$ we define
\beqq
\square_{\bar{k}}: = \{ j\in \N_0 :\quad  \supp\psi_{j}\cap\supp\varphi_{\bar{k}} \not=\emptyset\}
\eeqq
and $j\in \N_0$ 
\beqq
\Delta_{j}: = \{ \bar{k}\in \N_0^d :\quad  \supp\psi_{j}\cap\supp\varphi_{\bar{k}} \not=\emptyset\}.
\eeqq
The condition  
$  \supp\psi_{j}\cap\supp\varphi_{\bar{k}} \not=\emptyset $ 
implies  
\be \label{ws-02}
\max_{i=1, \ldots \, , d}\,  k_i-1\leq j\leq \max_{i=1,\ldots\, , d}\,  k_i+1.
\ee
Consequently we obtain
\be\label{ct2}
 |\square_{\bar{k}}| \asymp 1 \, , \ \ \bar{k} \in \N_0^d\, 
\qquad\text{and}\qquad
  |\Delta_{j}| \asymp (1+j)^{d-1}\, , \ \ j \in \N_0\, .  
\ee
By definition we have
\be\label{ct1b}
\psi_j(x)=\sum_{\bar{k}\in \Delta_{j}}\varphi_{\bar{k}}(x)\psi_j(x) \, , \qquad x \in \R\, .
\ee
{\it Step 2.}  Let $t>0$ and let $u=\min(1,p)$. Employing \eqref{ct1b} we find
\begin{equation}
\begin{split}\nonumber
\| f|B^{t}_{p,q}(\R)\|^q  = &\, \sum_{j=0}^{\infty}2^{jtq}\Big \|\sum_{\bar{k}\in \Delta_{j}}\gf ^{-1}\varphi_{\bar{k}} \psi_{j}\gf f\Big|L_{p}(\R)\Big\|^q 
\\
\leq &\,  \sum_{j=0}^{\infty}2^{jtq}\Big(\sum_{\bar{k}\in \Delta_{j}} \|\gf ^{-1}\varphi_{\bar{k}}\psi_j\gf f|L_p(\R)\|^u\Big)^{q/u}.
\end{split}
\end{equation}
Using \eqref{ct3} it follows
\beq\label{ct5}
\qquad\quad \| f|B^{t}_{p,q}(\R)\|^q \leq  C\,  \sum_{j=0}^{\infty}\bigg(\sum_{\bar{k}\in \Delta_{j}}\Big[2^{(j-|\bar{k}|)t} 2^{|\bar{k}|t}
\|\gf ^{-1}\varphi_{\bar{k}}\gf f|L_p(\R)\|\Big]^u\bigg)^{q/u}.
\eeq
If $\frac{q}{u}\leq 1$ then we have

\begin{equation}
\begin{split}
\| f|B^{t}_{p,q}(\R)\|^q 
&\leq C \, \sum_{j=0}^{\infty}\sum_{\bar{k}\in \Delta_{j}} 2^{(j-|\bar{k}|)tq}2^{|\bar{k}|tq}\|\gf ^{-1}\varphi_{\bar{k}}\gf f|L_p(\R)\|^q\, \\
&\leq  c_1 \, \sum_{\bar{k}\in \N_0^d}\sum_{j\in \square_{\bar{k}}}2^{|\bar{k}|tq}\|\gf ^{-1}\varphi_{\bar{k}}\gf f|L_p(\R)\|^q . \label{ct8}
\end{split}
\end{equation}
The  last inequality is due to $2^{(j-|\bar{k}|)tq}\leq c_2$ since $t>0$ and $j-1\leq \max_{i=1, \, \ldots\, ,d} k_i\leq j+1$, see \eqref{ws-02}.
In the case $\frac{q}{u}>1$ we use H\"older's inequality with $1 = \frac{u}q + (1-\frac uq)$.   \eqref{ct5} implies 
\beqq
\| f|B^{t}_{p,q}(\R)\|^q 
\,\leq\,C \, \sum_{j=0}^{\infty}\sum_{\bar{k}\in \Delta_{j}}2^{|\bar{k}|tq}\|\gf ^{-1}\varphi_{\bar{k}}\gf f|L_p(\R)\|^q
\Big(\sum_{\bar{k}\in \Delta_{j}}\big[2^{(j-|\bar{k}|)t} \big]^{\frac{q}{q-u}}\Big)^{\frac{q-u}{u}} .
\eeqq
Observe, for $t>0$ we have 
\[
\sup_{j \in \N_0}\, 
 \Big(\sum_{\bar{k}\in \Delta_{j}}\big[2^{(j-|\bar{k}|)t} \big]^{\frac{q}{q-u}}\Big)^{\frac{q-u}{u}} <\infty \, ,
\]
see \eqref{ws-02}.
Hence
\beq
\| f|B^{t}_{p,q}(\R)\|^q 
\,\leq \, c_3 \, \sum_{\bar{k}\in \N_0^d}\sum_{j\in \square_{\bar{k}}}2^{|\bar{k}|tq}\|\gf ^{-1}\varphi_{\bar{k}}\gf f|L_p(\R)\|^q  . \label{ct9}
\eeq
Finally, from \eqref{ct8}, \eqref{ct9} together with $\square_{\bar{k}}\asymp 1$  we conclude
\beqq
\| f|B^{t}_{p,q}(\R)\|^q \,\leq \,c_4\, \sum_{\bar{k}\in \N_0^d}2^{|\bar{k}|tq}\|\gf ^{-1}\varphi_{\bar{k}}\gf f|L_p(\R)\|^q.
\eeqq
This proves \eqref{ws-04}.\\
{\it Step 3.} Let $t=0$.
\\
{\it Substep 3.1.} First we assume that $q\leq \min(p,1)$. From \eqref{ct5} with $t=0$ we have
\begin{equation}
\begin{split}\nonumber
\| f|B^{0}_{p,q}(\R)\|^q 
&\leq \, c_1\, \sum_{j=0}^{\infty}\sum_{\bar{k}\in \Delta_{j}} \|\gf ^{-1}\varphi_{\bar{k}}\gf f|L_p(\R)\|^q\\
&= \, c_1 \, \sum_{\bar{k}\in \N_0^d}\sum_{j\in \square_{\bar{k}}}\|\gf ^{-1}\varphi_{\bar{k}}\gf f|L_p(\R)\|^q.
\end{split}
\end{equation}
Since $\square_{\bar{k}}\asymp1$ we obtain
\beqq
\| f|B^{0}_{p,q}(\R)\|^q  \,\leq \, c_2\, \| f|S^{0}_{p,q}B(\R)\|^q\, .
\eeqq 
{\it Substep 3.2.}  Let $1 < p < \infty$ and  $0 <  q \leq \min(2,p)$. Our main tool will be the following Littlewood-Paley assertion.
With $1 < p< \infty$ there exist positive constants $A,B$ such that
\be\label{LP}
A \, \|\, f \, |L_p (\R)\|\le 
\Big\| \Big(\sum_{k \in \N_0^d} |\gf ^{-1}\varphi_{\bar{k}}\gf f|^2\Big)^{1/2}\Big|L_p (\R)\Big\|  \le B \, \|\, f \, |L_p (\R)\|
\ee
holds for all $f \in L_p (\R)$, see Lizorkin \cite{Liz,Li} or Nikol'skij \cite[1.5.6]{Ni}. 
This will be applied to $f$ replaced by $\gf ^{-1} \psi_{j}\gf f$.
We proceed as in Step 1. Employing \eqref{ct1b} and \eqref{LP} we find
\begin{equation}\nonumber
\begin{split}
\| f|B^{0}_{p,q}(\R)\|^q & = \, \sum_{j=0}^{\infty} \Big\|\gf ^{-1} \psi_{j}\gf f
\Big|L_{p}(\R)\Big\|^q 
\\
&\leq \, \frac{1}{A^q}  \, \sum_{j=0}^{\infty}
\Big\| \Big(\sum_{\bar{k}\in \Delta_{j}} |\gf ^{-1}\varphi_{\bar{k}}\psi_j\gf f| ^2 \Big)^{1/2}\Big|L_p(\R)\Big\|^{q}.
\end{split}
\end{equation}
Because of $\| \, \cdot \, |L_p(\ell_2)\|\le \| \, \cdot \, |\ell_{\min(2,p)} (L_p)\| \le \| \, \cdot \, |\ell_{q} (L_p)\|$ 
we deduce
\begin{equation}\nonumber
\begin{split}
\| f|B^{0}_{p,q}(\R)\|^q &\le \,
\frac{1}{A^q}  \, \sum_{j=0}^{\infty}  \sum_{\bar{k}\in \Delta_{j}} 
\|\gf ^{-1}\varphi_{\bar{k}}\psi_j\gf f|L_p (\R)\|^q\\
&\le \,
c \, \sum_{j=0}^{\infty} \sum_{\bar{k}\in \Delta_{j}} \|\gf ^{-1}\varphi_{\bar{k}}\gf f|L_p (\R)\|^q
\end{split}
\end{equation}
where we used in the last step \eqref{ct3}.
As in Step 1 we can continue  the estimate by changing the order of summation and using $|\square_{\bar{k}}|\asymp 1$.
\qed


\subsection{Proof of Theorem \ref{besov1} -- sufficiency}
\label{proof1}


{\it Step 1.} Let us prove \eqref{ws-03} in case $t > \max (0, \frac 1p -1)$. We put $u:= \min (1,p)$. From we have
\be\label{ct1}
\varphi_{\bar{k}}(x) = \sum_{j\in \square_{\bar{k}}}\psi_j(x)\varphi_{\bar{k}}(x)\, , \qquad x \in \R\, .
\ee
This identity  yields
\beqq
\|f|S^{t}_{p,q}B(\R)\|^q 
\, = \,\sum_{\bar{k}\in \N_0^d}2^{|\bar{k}|tq}\Big\|\sum_{j\in \square_{\bar{k}}}\gf ^{-1}\psi_j \varphi_{\bar{k}} \gf f|L_{p}(\R)\Big\|^q.
\eeqq
Applying $|a+b|^u \le a^u + b^u$ and \eqref{ct4} we find
\begin{equation}
\begin{split}\nonumber
\|f|S^{t}_{p,q}B(\R)\|^q  \leq &   \sum_{\bar{k}\in \N_0^d}2^{|\bar{k}|tq}\Big(\sum_{j\in \square_{\bar{k}}}\big\|\gf ^{-1}\psi_j \varphi_{\bar{k}} \gf f|L_{p}(\R)\big\|^u\Big)^{q/u}\\
 \leq &\, C\,  \sum_{\bar{k}\in \N_0^d}2^{|\bar{k}|tq}\Big(\sum_{j\in \square_{\bar{k}}}\big(2^{(jd-|\bar{k}|)(\frac{1}{u}-1)}\|\gf ^{-1}\psi_j \gf f|L_p(\R)\| \big)^u\Big)^{q/u}.
\end{split}
\end{equation}
Because of \eqref{ct2} this implies
\beqq
\|f|S^{t}_{p,q}B(\R)\|^q \, \leq \, c\, \sum_{\bar{k}\in \N_0^d}\sum_{j\in \square_{\bar{k}}}2^{|\bar{k}|tq} 2^{(jd-|\bar{k}|)(\frac{1}{u}-1)q}\|\gf ^{-1}\psi_j \gf f|L_p(\R)\|^q\, .
\eeqq
Consequently 
\beqq
\|f|S^{t}_{p,q}B(\R)\|^q \, \le \, c\,  \sum_{j=0}^{\infty}   \, 2^{jdtq}\|\gf ^{-1}\psi_j \gf f|L_p(\R)\|^q\,          \sum_{\bar{k}\in \Delta_{j}} 2^{(jd-|\bar{k}|)(\frac{1}{u}-1-t)q}\, .
\eeqq
It is easily derived from \eqref{ws-02} and the restriction  $ t>\frac{1}{u}-1$ that  
\[
\sup_{j \in \N_0} \, \sum_{\bar{k}\in \Delta_{j}} 2^{(jd-|\bar{k}|)(\frac{1}{u}-1-t)q} < \infty \, .
\]
Hence
\beqq
\|f|S^{t}_{p,q}B(\R)\|^q \, \leq \,  c_1 \, \sum_{j=0}^{\infty}2^{jdtq}\|\gf ^{-1}\psi_j \gf f|L_p(\R)\|^q
\eeqq
follows.
\\
{\em Step 3.} Let $t=0$, $1 < p \le \infty$ and $\max(2,p)\leq q\leq \infty$. We shall argue by duality.
We have  
\[
S^0_{p',q'}B(\R) \hookrightarrow B^0_{p',q'}(\R)\,,
\]
see Theorem \ref{besov4}. \ref{dual1}(i) can be used to prove the claim. 
\\
{\em Step 4.} Let $0<  p \le  1$, $t= \frac 1p -1$ and $q=\infty$.
Applying \eqref{ct1} we find
\begin{equation}
\begin{split}\nonumber
\|f|S^{\frac{1}{p}-1}_{p,\infty}B(\R)\| = &\,\sup_{\bar{k}\in \N_0^d} 2^{|\bar{k}|(\frac{1}{p}-1)} \, 
\|\gf ^{-1}\varphi_{\bar{k}}\gf f|L_{p}(\R)\|\\
 = &\,\sup_{\bar{k}\in \N_0^d} \, 2^{|\bar{k}|(\frac{1}{p}-1)} \Big\|\sum_{j\in \square_{\bar{k}}}\gf ^{-1}\psi_j \varphi_{\bar{k}} \gf f|
L_{p}(\R)\Big\| .
\end{split}
\end{equation}
Making use of \eqref{ct4},  this implies
\begin{equation}
\begin{split}\nonumber
\|f|S^{\frac{1}{p}-1}_{p,\infty}B(\R)\|^p 
& \leq  \,  \sup_{\bar{k}\in \N_0^d} 2^{|\bar{k}|(\frac{1}{p}-1)p} \, 
\Big(\sum_{j\in \square_{\bar{k}}} \, \big\|\gf ^{-1}\psi_j \varphi_{\bar{k}} \gf f|L_{p}(\R)\big\|^p \Big)
\\
& \leq  c_1\,  \sup_{\bar{k}\in \N_0^d} \, 2^{|\bar{k}|(\frac{1}{p}-1)p} \, 
\Big(\sum_{j\in \square_{\bar{k}}} \, 2^{(jd-|\bar{k}|)(\frac{1}{p}-1)p}\,  \|\gf ^{-1}\psi_j \gf f|L_p(\R)\|^p  
\Big)\\
& =  c_1 \sup_{\bar{k}\in \N_0^d}\,  \Big(\sum_{j\in \square_{\bar{k}}} \, 2^{jd (\frac{1}{p}-1)p} \, 
\|\gf ^{-1}\psi_j \gf f|L_p(\R)\|^p  \Big).
\end{split}
\end{equation}

Taking into account \eqref{ct2}, we obtain 
\beqq
\|f|S^{\frac{1}{p}-1}_{p,\infty}B(\R)\| \, \leq \, c_2 \, \sup_{j\in \N_0}\,    2^{jd (\frac{1}{p}-1)}\, 
\|\gf ^{-1}\psi_j \gf f|L_p(\R)\| \, .
\eeqq
The proof is complete. \qed


\subsection{Proof of Theorem \ref{besov4} -- necessity}
\label{proof4}


\begin{lemma}\label{q<=2}
{\rm (i)} Let $0<p<\infty$ and $0<q\leq \infty$. Then the embedding 
  $$S^0_{p,q}B(\R)\hookrightarrow B^0_{p,q}(\R)\qquad \text{ implies }\qquad q\leq \min(2,p).$$
{\rm (ii)} Let  $0<q\leq \infty$. Then the embedding 
  $$S^0_{\infty,q}B(\R)\hookrightarrow B^0_{\infty,q}(\R)\qquad \text{ implies }\qquad q\leq 1.$$  
\end{lemma}

\begin{proof}
{\em Step 1.} We prove (i).\\
{\em Substep 1.1.} We show necessity of $q\leq 2$. 
Temporarily we assume $1 <p< \infty$. We use our test functions from Example 1, see \eqref{ws-example1}.
The embedding $S^0_{p,q}B(\R)\hookrightarrow B^0_{p,q}(\R)$ implies
the existence of a constant $c$ such that 
\[
\Big(\sum_{j=1}^{\ell} |a_j|^2\Big)^{1/2} \asymp 
 \|\, f_\ell\, |B^{0}_{p,q}(\R)\|  \le c \, \|\, f_\ell\, |S^{t}_{p,q}(\R)\|\asymp \Big(\sum_{j=1}^\ell   \, |a_j|^q  \Big)^{1/q}
\]
where $c$ does not depend on $\ell$ and $(a_j)_j$,
see \eqref{ws-09} and \eqref{ws-10}.
This requires $q \le 2$.
\\
Now we turn to $0 < p \le 1$. Again we shall work with Example 1. For any such $p$ there exists some real number $\Theta \in (0,1)$ such that
\[
 \frac 23 = \frac{1-\Theta}{p} + \frac \Theta 2\, .
\]
Lyapunov's inequality 
\[
\|\, h \,|L_{3/2}([-\pi,\pi]^2)\|
\le  \|\, h\, |L_p([-\pi,\pi]^2)\|^{1-\Theta} 
\, \|\, h \, |L_2([-\pi,\pi]^2)\|^\Theta \, ,
\]
valid for all $h \in L_p([-\pi,\pi]^2) \cap L_2([-\pi,\pi]^2) $, in combination with 
the Littlewood-Paley characterization of $L_{3/2}$ and $L_2$, leads us to 
\[
 \Big(\sum_{j=1}^{\ell} |a_j|^2\Big)^{1/2} \le c\, 
 \Big\|\, \sum_{j=1}^{\ell} a_j \, e^{\frac{7}{8}i(2^\ell x_1 + 2^j x_2)} \,  \Big|L_p([-\pi,\pi]^2)\Big\|^{1-\Theta} 
 \Big(\sum_{j=1}^{\ell} |a_j|^2\Big)^{\Theta/2}
\]
with $c$ independent of $\ell$ and $(a_j)_j$.
Hence
\[
 \Big(\sum_{j=1}^{\ell} |a_j|^2\Big)^{1/2} \le c \, \Big\|\, \sum_{j=1}^{\ell} a_j \, e^{\frac{7}{8}i(2^\ell x_1 + 2^j x_2)} \,  \Big|L_p([-\pi,\pi]^2)\Big\|\, .
\]
Taking into account \eqref{ws-08} we can argue as in case $1 < p< \infty$.\\ 
{\em Substep 1.2.} We show necessity of  $q\leq p$. Therefore we use Example 3. 
In case  $1<p<\infty$ we choose
$a_{\bar{k}} = 2^{|\bar{k}| (\frac 1p -1)}$. Then almost immediately we can conclude $q\le p$.
\\
Now we turn to the remaining cases.
Assume that there exist $0<p\leq 1$ and $p<q\leq 2$ such that
\beqq
S^{0}_{p,q}B(\R)\hookrightarrow B^0_{p,q}(\R)\, .
\eeqq
In this situation we may choose  a triple $(p_1,q_1,\Theta)$ such that 
\beqq
1<p_1<q_1\leq 2,\quad \Theta\in (0,1),\quad \frac{1}{p_1}=\frac{\Theta}{p}+\frac{1-\Theta}{2}\quad \text{and} \quad \frac{1}{q_1}=\frac{\Theta}{q}+\frac{1-\Theta}{2}.
\eeqq
Then it follows from Proposition \ref{inter2} that
\beqq
S^0_{p_1,q_1}B(\R)=[S^0_{p,q}B(\R),S^0_{2,2}B(\R)]_{\Theta}
\eeqq
and
\beqq
B^0_{p_1,q_1}(\R)=[B^0_{p,q}(\R),B^0_{2,2}(\R)]_{\Theta}.
\eeqq
Proposition \ref{inter1}  yields
\beqq
S^0_{p_1,q_1}B(\R)\hookrightarrow B^0_{p_1,q_1}(\R).
\eeqq
But this is a contradiction to Example 3.
\\
{\em Step 2.}
To prove (iii)  we use Example 2, see \eqref{ws-example2}.
The embedding $S^0_{\infty,q}B(\R)\hookrightarrow B^0_{\infty,q}(\R)$ implies
the existence of a constant $c$ such that 
\[
\sum_{j=1}^{\ell} |a_j| =  
 \|\, f_\ell\, |B^{0}_{\infty,q}(\R)\|  \le c \, \|\, f_\ell\, |S^{0}_{\infty,q}(\R)\| \asymp \Big(\sum_{j=1}^\ell  \, |a_j|^q  \Big)^{1/q}
\]
where $c$ does not depend on $\ell$ and $(a_j)_j$,
see \eqref{ws-11} and \eqref{ws-12}.
Choosing $a_j =1$ it is obvious that this can happen only if $q\le 1$.
\end{proof}

\noindent
\textsc{Proof of Theorem \ref{besov4}}.
{\em Step 1}.
Let $t=0$. Then the necessity of $q \le \min (p,2)$ if $p< \infty$ 
and of $q\le 1$ if $p= \infty$ follows from  Lemma \ref{q<=2}.
\\
{\em Step 2.} It remains to deal with $t<0$.
We shall employ  complex interpolation.
\\
{\em Substep 2.1.} We assume that $0 < p< \infty$ and $1<q<\infty$. Let $\Theta =1 /2$. From Proposition \ref{inter2}  we obtain 
\[
B^{0}_{p,q_1}(\R)=[B^{-t}_{p,\infty}(\R),\, B^{t}_{p,q}(\R)]_{\Theta}
\]
and
\[
S^{0}_{p,q_1}B(\R)=[S^{-t}_{p,\infty}B(\R),\, S^{t}_{p,q} B(\R)]_{\Theta}, \qquad \frac{1}{q_1} =  \frac{1}{2q}\, .
\]
Theorem \ref{besov4} and the assumption  $S^t_{p,q}B(\R)\hookrightarrow B^t_{p,q}(\R)$ lead to
the conclusion $S^0_{p,q_1}B(\R)\hookrightarrow B^0_{p,q_1}(\R)$. In view of  Lemma \ref{q<=2} this embedding implies  
$2q = q_1 \le \min (p,2)$. This is a contradiction.
\\
{\em Substep 2.2.} Let $0 < p< \infty$ and $q=\infty$. This time we use 
\[
B^{0}_{p,4}(\R)=[B^{-t}_{p,2}(\R),\, B^{t}_{p,\infty}(\R)]_{\Theta}
\]
and
\[
S^{0}_{p,4}B(\R)=[S^{-t}_{p,2}B(\R),\, S^{t}_{p,\infty} B(\R)]_{\Theta}.
\]
Here $\Theta=1/2$. Now Proposition \ref{inter1} implies $S^{0}_{p,4}B(\R)\hookrightarrow B^{0}_{p,4}(\R)$. This contradicts  Theorem \ref{besov4}.\\
{\em Substep 2.3.}
Let $0 < p< \infty$ and $0 < q \le 1$. We argue as in the previous step.
Therefore we shall use
\[
B^{0}_{p_1,q_1}(\R)=[B^{s}_{\infty,\infty}(\R),\, B^{t}_{p,q}(\R)]_{\Theta}
\]
and
\[
S^{0}_{p_1,q_1}B(\R)=[S^{s}_{\infty,\infty}B(\R),\, S^{t}_{p,q} B(\R)]_{\Theta},
\]
where we need $(1-\Theta) s + \Theta t =0$,  
\[
\frac{1}{p_1} =  \frac{\Theta}{p} \qquad \mbox{and}\qquad \frac{1}{q_1} = \frac{\Theta}{q}\, . 
\]
By choosing $\Theta$ small we arrive at $q_1 >2$. This contradicts Lemma \ref{q<=2}.
\\
{\em Substep 2.4.} We assume that  $p = \infty$ and $0 < q\le \infty$. Proposition \ref{inter3} yields
\[
\accentset{\diamond}{B}^{0}_{\infty,q_1}(\R)=[B^{s}_{\infty,\infty}(\R),\, B^{t}_{\infty,q}(\R)]_{\Theta}
\]
and
\[
\accentset{\diamond}{S}^{0}_{\infty,q_1}B(\R)=[S^{s}_{\infty,\infty}B(\R),\, S^{t}_{\infty,q} B(\R)]_{\Theta},
\]
where  $(1-\Theta) s + \Theta t =0$ and   $ \frac{1}{q_1} = \frac{\Theta}{q} $.
We choose $\Theta$ small  
enough such that $q_1>1$.  
Then, as a conclusion of Theorem \ref{besov4} and Propositions \ref{inter1}, \ref{inter2},
we get $\accentset{\diamond}{S}^{0}_{\infty,q_1}B(\R) \hookrightarrow \accentset{\diamond}{B}^{0}_{\infty,q_1}(\R)$.
Now we argue as in Step 3 of the proof of Lemma \ref{q<=2} by taking into account that our test functions from Example 2
are elements of  $\accentset{\diamond}{B}^{0}_{\infty,q_1}(\R) \cap \accentset{\diamond}{S}^{0}_{\infty,q_1}B(\R)$.
\qed


\subsection{Proof of Theorem \ref{besov1} -- necessity}
\label{proof3}


By means of the same arguments as used in proof of Lemma \ref{q<=2} the following dual assertion can be proved.

\begin{lemma}\label{q>=2}
Let $1<p\leq\infty$ and $0<q\leq \infty$. Then the embedding 
  $$B^0_{p,q}(\R)\hookrightarrow S^0_{p,q}B(\R)\qquad \text{ implies }\qquad q\ge \max(p,2).$$
\end{lemma}

\noindent
\textsc{Proof of Theorem  \ref{besov1}}.
{Step 1}. Let $0 <p \le 1$ and $t=\frac 1p -1$. Assume that there is some $q<\infty$ such that
$B^{td}_{p,q} (\R) \hookrightarrow S^{t}_{p,q}B (\R)  $ holds.
Then Proposition \ref{dual1} yields
$S^{0}_{\infty,q'} B(\R) \hookrightarrow B^{0}_{\infty,q'}(\R)$.
In view of Lemma   \ref{q<=2} this implies $q'\le 1$, hence $q= \infty$.
\\
{\em Step 2.} The necessity of the restrictions in case $t=0$ follows by Lemma \ref{q>=2}.
\\
{\em Step 3.} It remains to deal with $t< \max (0, \frac 1p -1)$.
\\
{\em Step 3.1.} Let $0 < p \leq \infty$ and $t<0$.
We employ the test functions from Example 1.
Choosing $a_j = \delta_{j,\ell}$ in \eqref{ws-example1}, we find
\be
 \| \, f_\ell \, |B^{td}_{p,q} (\R)\| = 2^{\ell td}\, \| \, \gf ^{-1} g \, |L_p (\R)\| \label{k-1}
\ee
and 
\be
 \| \, f_\ell \, |S^{t}_{p,q} B(\R)\| = 2^{\ell t}\, \| \, \gf ^{-1} g \, |L_p (\R)\|\, ,  \label{k-2}
\ee
see \eqref{ws-08} and \eqref{ws-10}.
With $\ell \to \infty$ it becomes clear that 
\[
B^{td}_{p,q} (\R) \hookrightarrow S^{t}_{p,q} B(\R)
\]
can not hold.
\\
{\em Step 3.2.} Let $0 < p < 1$ and $0 \le t< \frac 1p -1$.
We assume $B^{td}_{p,q} (\R) \hookrightarrow S^{t}_{p,q} B(\R)$. Proposition \ref{dual1} yields 
\[
S^{-t+\frac 1p -1}_{\infty,q'}B(\R) \hookrightarrow B_{\infty,q'}^{d(-t+ \frac 1p -1)} (\R) \, .
\]
Since $d(-t + \frac 1p -1 ) > -t + \frac 1p -1$ it is enough to use 
$g_k (x) : = e^{ikx}$, $k \in \Z$, as test functions to disprove this embedding.
\qed


\subsection{Proof of Propositions \ref{besov14}, \ref{besov15}}
\label{proof6}


\textsc{Proof of Proposition \ref{besov14}.} 
Part (i) in case $1 \le q\le \infty$ follows by duality from Theorem \ref{besov4}, see Proposition \ref{dual1}. 
To cover also the cases $0 < q < 1$ we argue as in proof of Theorem \ref{besov1} (sufficiency)  replacing $B^{td}_{p,q}(\R)$ by $B^t_{p,q}$.
Observe in this connection that  
\[
\sup_{j \in \N_0} \, \sum_{\bar{k}\in \Delta_{j}} 2^{(|\bar{k}|-j)tq} < \infty \, .
\]
Parts (ii)-(iv) are immediate consequences of Theorems \ref{besov4}, \ref{besov1}.
\\
Now we turn to the proof of (v). Theorem \ref{besov4} yields 
$S^{t}_{p,q} B(\R) \not \hookrightarrow B^{t}_{p,q} (\R)$.
It remains to prove
$B^{t}_{p,q} (\R) \not \hookrightarrow S^{t}_{p,q} B(\R)$.
Therefore we employ Example 1 with $a_j := 2^{-jt}$, $j\in \N$.
From \eqref{ws-08} it follows
\begin{equation}
\begin{split}\nonumber
\|\, f_\ell\, |B^{t}_{p,q}(\R)\| 
\le &\, c\,  2^{\ell t} \, \Big\|\, \sum_{j=1}^{\ell} a_j \, e^{\frac{7}{8}i(2^\ell x_1 + 2^j x_2)} \,  \Big|L_p([-\pi,\pi]^2)\Big\| 
\\
\le &\, c  2^{\ell t} \, \Big( \sum_{j=1}^{\ell} |a_j|^p\Big)^{1/p} \asymp 1 
\, .
\end{split}
\end{equation}
On the other hand, \eqref{ws-10} yields
\[
\|\, f_\ell\, |S^{t}_{p,q}(\R)\|   =   \big\|  \, \gf^{-1}  g\, |L_p(\R)\big\| \, \Big(\sum_{j=1}^\ell  2^{jtq} \, |a_j|^q  \Big)^{1/q}
\asymp \ell^{1/q}
\, .
\]
If we assume $B^{t}_{p,q} (\R)  \hookrightarrow S^{t}_{p,q} B(\R)$, $q<\infty$,  then this leads to a contradiction. In case $q=\infty$ we employ the same type of argument but choose $a_j:= 2^{-jt}j^{-1}$, $j \in \N$. \qed
\\
{~}\\
\noindent
\textsc{Proof of Proposition \ref{besov15}.} 
To prove (i)  we follow the arguments used in proof of Theorem \ref{besov4} (sufficiency) by replacing 
$B^t_{p,q}(\R)$ by $B^{td}_{p,q}(\R)$.
\\
Concerning (ii)-(iv), observe that $B^{td}_{p,q} (\R) \not \hookrightarrow S^{t}_{p,q} B(\R)$ follows from Theorem \ref{besov1}.
Now we split our investigations into two cases: $0 < t < \frac 1p -1$, $t=0$.
\\
{\em Step 1.} Let $0 <p<1$ and $0< t < \frac 1p -1$.
$S^{t}_{p,q} B(\R) \not \hookrightarrow B^{td}_{p,q} (\R)$ follows from 
Example 1 with $a_j := \delta_{j,\ell}$, see \eqref{k-1} and \eqref{k-2}.
\\
{\em Step 2.} Let $0 <p<1$ and $t=0$.
Theorem \ref{besov4} yields
\[
 S^{0}_{p,q} B(\R)  \hookrightarrow B^{0}_{p,q} (\R)\qquad \Longleftrightarrow \qquad 0 < q \le p\, .
\]
Hence, in case $q>p$ the spaces  $S^{0}_{p,q} B(\R)$ and $B^{0}_{p,q} (\R)$ are not comparable. \qed


\subsection{Proofs of the optimality assertions}
\label{proof5}


First we recall some well-known results about embeddings of Besov spaces.
Let $0 < p \le p_0 \le \infty$. Then 
$B^{t}_{p,q} (\R) \hookrightarrow B^{t_0}_{p_0,q_0} (\R) $ holds if and only if
either 
\[
 t_0 - \frac{d}{p_0} <  t- \frac dp  \qquad \mbox{and \quad $q,q_0$ are arbitrary}
\]
or 
\[
  t_0 - \frac{d}{p_0} = t- \frac dp  \qquad \mbox{and} \qquad  q \le q_0 \, .
\]
This result has a certain history. For the first time it has been proved by Taibleson in his series of papers  \cite{ta1}-\cite{ta3}, but see also \cite[2.7.1]{Tr83} and \cite{sitr}.
In case of the Besov spaces of dominating mixed smoothness the following is known. Again we suppose
$0 < p \le p_0 \le \infty$. Then 
$S^{t}_{p,q} B(\R) \hookrightarrow S^{t_0}_{p_0,q_0} B(\R) $ holds if and only if
either 
\[
 t_0 - \frac{1}{p_0} <  t- \frac 1p  \qquad \mbox{and \quad $q,q_0$ are arbitrary}
\]
or 
\[
  t_0 - \frac{1}{p_0} = t- \frac 1p  \qquad \mbox{and} \qquad  q \le q_0 \, .
\]
We refer to \cite{SS} and \cite{HV}.
\\
{~}\\
\noindent
\textsc{ Proof of Theorem \ref{besov5}}. 
Assuming $S^{t_0}_{p_0,q_0}B (\R) \hookrightarrow B^{t}_{p,q} (\R) $
Lemma \ref{p0p1}  implies $p_0 \le p$.
Next we apply Example 4 to derive some relations between $p_0$ and $p$.
We choose $ a_{j} := \delta_{j,\ell}\ \,, j=1,...,\ell\, . 
$
Then it follows 
\[
 \| \, f_\ell\, | B^{t}_{p,q} (\R)\| = C \, 2^{\ell (t+1-1/p)}\
\qquad\text{and}\qquad
 \| \, f_\ell\, | S^{t_0}_{p_0,q_0} (\R)\| = C  2^{\ell (t+1-1/p_0)}\,.
\]
The assumed embedding implies 
\beqq
 t + 1 - \frac 1p \le t_0 + 1 - \frac{1}{p_0} \, .
\eeqq
In case $t - \frac 1p = t_0 - \frac{1}{p_0}$ we employ Example 4 again  with $a_j := 2^{-j(t+1-1/p)}$, $j=1 ,\ldots \, , \ell$.
As a consequence of the embedding we derive 
$q_0 \le q$.
This implies
\[
S^{t_0}_{p_0,q_0}B (\R)  \hookrightarrow S^{t}_{p,q}B (\R) \, , 
\]
see the above comments to embeddings of Besov spaces.\qed
\\
{~}\\
\noindent
{\textsc{Proof of Theorem \ref{besov2}}}. 
Assuming $B^{t_0}_{p_0,q_0} (\R) \hookrightarrow S^{t}_{p,q}B (\R) $
Lemma \ref{p0p1} implies $p_0 \le p$.
Next we employ Example 5 where the $a_j:= \delta_{j,\ell}$, $j=1, \ldots \, ,\ell$.
We obtain
\[
 \| \, f_\ell \, | S^{t}_{p,q}B (\R)\| = C \, 2^{\ell d (t + 1 - \frac 1p) }
\qquad \text{and }\qquad
 \| \, f_\ell \, | B^{t_0}_{p_0,q_0} (\R)\| = C \, 2^{\ell d ( \frac{t_0}{d} + 1 - \frac{1}{p_0}) }
\]
with $C>0$ independent of $\ell$.
The embedding $B^{t_0}_{p_0,q_0} (\R)\hookrightarrow  S^{t}_{p,q}B (\R) $ yields
\[
  d \Big( \frac{t_0}{d} + 1 - \frac{1}{p_0}\Big)  \geq  d \Big(t + 1 - \frac 1p\Big) \qquad \Longleftrightarrow \qquad t_0- \frac{d}{p_0}  \geq  d t  - \frac dp \, .
\]
Now, if $t_0- \frac{d}{p_0}  =  d t  - \frac dp$, we apply Example 5, again with $a_j := 2^{-j(t_0+1-1/{p_0})}$, to obtain $q_0\leq q$. All together we conclude
\[
B^{t_0}_{p_0,q_0} (\R)  \hookrightarrow B^{td}_{p,q}B (\R) \, , 
\]
see the above comments on embeddings.
\qed

{~}\\
\noindent
{\textsc{ Proof of Theorem \ref{besov3}}}. 
Assuming $B^{t d}_{p,q} (\R) \hookrightarrow S^{t_0}_{p_0,q_0}B (\R) $
Lemma \ref{p0p1} implies $p \le p_0$.
Example 5 with  $a_j:= \delta_{j,\ell}$, $j=1, \ldots \, ,\ell$, yields 
\[
 \| \, f_\ell \, | S^{t_0}_{p_0,q_0}B (\R)\| = C \, 2^{\ell d (t_0 + 1 - \frac {1}{p_0}) }
 \qquad\text{and}\qquad
 \| \, f_\ell \, | B^{td}_{p,q} (\R)\| = C \, 2^{\ell d (t + 1 - \frac{1}{p}) }
\]
with $C>0$ independent of $\ell$.
The embedding $B^{td}_{p,q} (\R) \hookrightarrow S^{t_0}_{p_0,q_0} B(\R) $ implies
\[
  d \Big( t_0 + 1 - \frac{1}{p_0}\Big)  \le  d \Big(t + 1 - \frac 1p\Big) \qquad \Longleftrightarrow \qquad t_0- \frac{1}{p_0}  \le   t  - \frac 1p \, .
\]
Working with Example 5 in the case $t_0- \frac{1}{p_0}  =   t  - \frac 1p$, choose  $a_j := 2^{-j(t+1-1/{p})}$, 
we obtain $q\leq q_0$. Taking into account the above comments on embeddings  we arrive at
$S^{t}_{p,q} B(\R) \hookrightarrow S^{t_0}_{p_0,q_0} B(\R)$. \qed

\bibliographystyle{amsalpha}

\end{document}